\documentclass{amsart}

\usepackage{amssymb}
\usepackage{amsmath}
\usepackage{amsfonts}
\usepackage{geometry}
\usepackage{bbm}
\usepackage{hyperref}
\usepackage{tikz}
\usepackage{mathrsfs}
\usetikzlibrary{matrix,arrows,decorations.pathmorphing}

\newcounter{cprop}[section]

\newtheorem{theorem}[cprop]{Theorem}

\theoremstyle{plain}

\newtheorem{corollary}[cprop]{Corollary}

\newtheorem{lemma}[cprop]{Lemma}
\newtheorem{proposition}[cprop]{Proposition}

\numberwithin{equation}{section}

\theoremstyle{definition}
\newtheorem{definition}[cprop]{Definition}

\theoremstyle{remark}
\newtheorem{remark}[cprop]{Remark}

\newcommand{\E}{\mathbb{E}}
\renewcommand{\P}{\mathbb{P}}

\newcommand{\R}{\mathbb{R}}
\newcommand{\N}{\mathbb{N}}
\newcommand{\Z}{\mathbb{Z}}

\newcommand{\vertiii}[1]{{\left\vert\kern-0.25ex\left\vert\kern-0.25ex\left\vert #1 
    \right\vert\kern-0.25ex\right\vert\kern-0.25ex\right\vert}}

\begin{document}
\title[Dynamical theory for SDDE II]{A dynamical theory for singular stochastic delay differential equations II: Nonlinear equations and invariant manifolds}

\author{M. Ghani Varzaneh}
\address{Mazyar Ghani Varzaneh\\
Institut f\"ur Mathematik, Technische Universit\"at Berlin, Germany and Department of Mathematical Sciences, Sharif University of Technology, Tehran, Iran}
\email{mazyarghani69@gmail.com}

\author{S. Riedel}
\address{Sebastian Riedel \\
Institut f\"ur Mathematik, Technische Universit\"at Berlin, Germany and Weierstra{\ss}-Institut, Berlin, Germany}
\email{riedel@math.tu-berlin.de}
%

\renewcommand{\thefootnote}{\fnsymbol{footnote}} 
\renewcommand{\thefootnote}{\arabic{footnote}} 



\begin{abstract}
  Building on results obtained in \cite{GRS19}, we prove \emph{Local Stable and Unstable Manifold Theorems} for nonlinear, singular stochastic delay differential equations. The main tools are rough paths theory and a semi-invertible Multiplicative Ergodic Theorem for cocycles acting on measurable fields of Banach spaces obtained  in \cite{GR19}.
\end{abstract}

\maketitle

\section*{Introduction}

  The following article is a sequel to \cite{GRS19}. Our aim is to study stochastic delay differential equations (SDDEs) of the form
  \begin{align}\label{eqn:sdde_intro}
   dy_t = b(y_t,y_{t-r})\, dt + \sigma(y_t,y_{t-r})\, dB_t(\omega)
  \end{align}
  from a dynamical systems point of view. In \eqref{eqn:sdde_intro}, $r > 0$ denotes a time delay, $B$ is a multidimensional Brownian motion, $b$ is the drift and $\sigma$ the diffusion coefficient. Such equations are called \emph{(single) discrete time delay equations}.\footnote{The results in \cite{GRS19} and in the present article do also apply for vector fields depending on a finite number of time instances in the past, but we restrict ourselves to a single delay for the sake of simplicity.} The goal in the present article is to prove the existence of \emph{random invariant manifolds} for \eqref{eqn:sdde_intro}. Invariant manifolds are key objects in the theory of dynamical systems, both deterministic and random, and play a central role, for instance, in stochastic bifurcation theory \cite{KW83,Arn98,CLR01} and model reduction for stochastic differential equations \cite{DD07,DW14,CLW15a,CLW15b}.
  
  Although the equation \eqref{eqn:sdde_intro} can be easily solved with It\=o's theory of stochastic integration, studying its dynamical properties is a challenging task. In fact, the key object in the theory of random dynamical systems \cite{Arn98} is the \emph{cocyle} which is induced by a stochastic differential equation. However, Mohammed \cite{Moh86} showed that one can not expect that an equation of the form \eqref{eqn:sdde_intro} induces a continuous stochastic flow (cf. also \cite[Theorem 2.1]{MS97} and \cite[Theorem 0.2]{GRS19} for similar results), therefore it was believed that \eqref{eqn:sdde_intro} does, in general, not induce a cocycle. Without going too much into detail here, we want to mention that the source of trouble in \eqref{eqn:sdde_intro} is the diffusion coefficient $\sigma$ which is allowed to depend on the past. Equations where the delay only appears in the drift are easier to handle and their dynamical properties were studied, for instance, in \cite{MS90, MS96, MS97, MS03, MS04}. If the diffusion $\sigma$ is path-dependent in a smooth way, i.e. when
  \begin{align*}
   \sigma(y_t,y_{\cdot}) = \int_{-r}^0 \hat{\sigma}(y_t,y_{t+s})\, \mu(ds)
  \end{align*}
  for a regular measure $\mu$, the situation is also simpler and was considered, in parts, in the above mentioned references. The equation \eqref{eqn:sdde_intro} corresponds to $\mu$ being the (singular) Dirac measure $\delta_{-r}$ which is the reason for calling it a \emph{singular} stochastic delay equation. 
  
  One of our main results in \cite{GRS19} was that \eqref{eqn:sdde_intro} does indeed induce a cocycle. However, one has to pay a price: the spaces on which the cocycle map is defined will depend on the trajectory of the driving path $B(\omega)$. More precisely, if $(\Omega,\mathcal{F},\P,\theta)$ is a random dynamical system (cf. definition below), the cocycle $\varphi$ is a continuous map 
  \begin{align*}
   \varphi(n,\omega,\cdot) \colon E_{\omega} \to E_{\theta^n \omega}
  \end{align*}
  where $\{E_{\omega}\}_{\omega \in \Omega}$ is a familiy of Banach spaces. In the literature, these type of cocycles are not new and were already studied. For instance, they naturally appear when linearizing a stochastic differential equation on a manifold \cite[Section 4.2]{Arn98}. One key idea in \cite{GRS19} was to interprete \eqref{eqn:sdde_intro} as a random rough differential equation in the sense of Lyons \cite{Lyo98, NNT08, FH14}. Doing this, we showed in \cite{GRS19} that Gubinelli's spaces of controlled paths \cite{Gub04} are possible choices for $E_{\omega}$ when studying \eqref{eqn:sdde_intro}.
  
  A major result in smooth ergodic theory is the \emph{Multiplicative Ergodic Theorem} (MET) which provides a spectral theory for linear cocycles. In \cite{GRS19}, we proved that such a theorem holds in our framework. As a consequence, we could show that cocycles induced by linear equations of the form \eqref{eqn:sdde_intro} possess a \emph{Lyapunov spectrum}, an analogue to the set of eigenvalues of a matrix. In \cite{GR19}, we proved in a more abstract framework that an \emph{Oseledets splitting}, i.e. a decomposition of $E_{\omega}$ into a direct sum of $\varphi$-invariant spaces, can also be deduced. This was the basis to prove the existence of local stable and unstable manifolds.
  
  In this article, we harvest the fruit of our former work. In our main results, Theorem \ref{thm:stable_manifold} and Theorem \ref{thm:unstable_manifold}, we formulate sufficient conditions under which we can deduce the existence of local stable and unstable manifolds for equation \eqref{eqn:sdde_intro}. Let us mention that one difficulty in the unstable case is that the cocycle induced by \eqref{eqn:sdde_intro} is not invertible, which is natural for delay equations: solutions exist only forward in time. Therefore, we can not just apply the stable manifold theorem to the inverse cocycle as, for instance, in \cite{MS99}. To overcome this difficulty, we use the semi-invertible MET in \cite{GR19} to obtain the existence of unstable manifolds. Both theorems are formulated in a generality which allows them to be applied to equations which are driven by a much more general noise than Brownian motion, e.g. by semimartingales with stationary increments or by a fractional Brownian motion. 
  
  There are many invariant manifold theorems for stochastic differential equations. In the case of a finite dimensional state space, let us mention \cite{Car85, Box89, Wan95, MS99,KN18}. For infinite dimensional state spaces, invariant manifold theorems were proved by Mohammed and Scheutzow for a class SDDEs in \cite{MS04} and for different classes of stochastic partial differential equations in \cite{DLS03, DLS04, MZZ08, CDLS10, MZ10, GLS10b, CRD15, LNS18, CRD19, Nea19}. 
  
  The structure of the paper is as follows: In Section \ref{sec:sdde_nonlinear}, we study properties of rough delay differential equations. In particular, we prove their differentiability and provide bounds for the derivative. We furthermore study equations with a linear drift term. Section \ref{sec:inv_mfd_rrde} contains our main results. We introduce random fixed points for cocycles (\emph{stationary trajectories}) around which the invariant manifolds exist. The main results are formulated in Theorem \ref{thm:stable_manifold} and Theorem \ref{thm:unstable_manifold}. Subsection \ref{subsec:examles} contains examples of equations for which our theorems apply.

\subsection*{Preliminaries and notation}

In this section we collect some conventions, the notation and basic definitions which will be used throughout the paper. The notation coincides with the one used in \cite{GRS19}. \\ 
\begin{itemize}
 \item \emph{Differentiable} will always mean differentiability in Fr\'echet-sense.
 \item If not stated differently, $U$, $V$, $W$ and $\bar{W}$ will always denote finite-dimensional, normed vector spaces over the real numbers, with norm denoted by $|\cdot |$. The space $ L(U,W) $ consists of all bounded linear functions from $ U $ to $ W $ equipped with usual operator norm.
  \item Let $I$ be an interval in $\R$. A map $ m:I\rightarrow U $ will also be called a \emph{path}. For a path $m$, we denote its increment by $ m_{s,t}=m_{t}-m_{s} $ where by $ m_{t} $ we mean $ m(t) $. We set
\begin{align*}
\Vert m\Vert_{\infty;I}:=\sup_{s\in I}\vert m_{s}\vert
\end{align*}
and define the $ \gamma$-H\"older seminorm, $\gamma \in (0,1]$, by
\begin{align*}
\Vert m\Vert_{\gamma;I} := \sup_{s,t \in I; s \neq t} \frac{\vert m_{s,t}\vert}{\vert t-s\vert^{\gamma}}.
\end{align*}
For a general $2$-parameter function $m^{\#} \colon I \times I \to U$, the same notation is used. We will sometimes omit $I$ as subindex if the domain is clear from the context. 


\item By $C^n_b(W^2, \bar{W})$, we denote the space of bounded functions $\sigma \colon W \oplus W \to \bar{W}$ having $n$ bounded derivatives, $n \geq 0$. Often, we will omit domain and codomain and just write $C^n_b$. We set $\sigma_{x^n,y^m} := \frac{\partial^{n+m}}{\partial x^n \partial y^m} \sigma(x,y)$ for $n,m \geq 0$ and $\sigma_{x} := \sigma_{x^1,y^0}$, $\sigma_{y} := \sigma_{x^0,y^1}$. Dropping the subindex $b$ means dropping the boundedness assumption.

\end{itemize}

Next, we introduce notions from rough paths theory needed in this article. Most of them can be found in \cite{FH14}. We also review some of the concepts from \cite{NNT08} and \cite{GRS19} here.
\begin{itemize}
 \item  Let  $X \colon \R \to U$ be a locally $\gamma$-H\"older path, $\gamma \in (0,1]$. A \emph{L\'evy area} for $X$ is a continuous function
 \begin{align*}
  \mathbb{X} \colon \R \times \R \to U \otimes U
 \end{align*}
 for which the algebraic identity
 \begin{align*}
   \mathbb{X}_{s,t} = \mathbb{X}_{s,u} + \mathbb{X}_{u,t} + X_{s,u} \otimes X_{u,t} 
 \end{align*}
 is true for every $s,u,t \in \R$ and for which $\| \mathbb{X} \|_{2\gamma ; I} < \infty$ holds on every compact interval $I \subset \R$. If $\gamma \in (1/3,1/2]$ and $X$ admits L\'evy area $\mathbb{X}$, we call $ \mathbf{X}= \big{(}X, \mathbb{X}\big{)}$ a \emph{$\gamma$-rough path} and set $\| \mathbf{X}\|_{\gamma ; I} := \|X\|_{\gamma ; I} + \sqrt{\| \mathbb{X} \|_{2\gamma ; I}}$. A \emph{delayed L\'evy area} for $X$ is a continuous function
 \begin{align*}
   \mathbb{X}(-r) \colon \R \times \R \to U \otimes U
 \end{align*}
 for which the algebraic identity
\begin{align*}
   \mathbb{X}_{s,t}(-r)=\mathbb{X}_{s,u}(-r) + \mathbb{X}_{u,t}(-r)+X_{s-r,u-r}\otimes X_{u,t} 
 \end{align*}
 holds for every $s,u,t \in \R$ and for which we have $\| \mathbb{X}(-r) \|_{2\gamma ; I} < \infty$ on every compact interval $I \subset \R$. If $\gamma \in (1/3,1/2]$ and $X$ admits L\'evy- and delayed L\'evy area $\mathbb{X}$ and $\mathbb{X}(-r)$, we call $ \mathbf{X}= \big{(}X, \mathbb{X}, \mathbb{X}(-r)\big{)}$ a \emph{delayed $\gamma$-rough path with delay $r > 0$}. For an interval $[a,b] \subset \R$, we set 
 \begin{align*}
  \| \mathbf{X}\|_{\gamma ; [a,b]} := \|X\|_{\gamma ; [a,b]} + \|X\|_{\gamma ; [a-r,b-r]} + \sqrt{\| \mathbb{X} \|_{2\gamma ; [a,b]}} + \sqrt{\| \mathbb{X}(-r) \|_{2\gamma ; [a,b]}}.
 \end{align*}

 
 \item   Let $I = [a,b]$ be a compact interval. A path $m \colon I \to \bar{W}$ is a \emph{controlled path} based on $X$ on the interval $I$ if there exists a $\gamma$-H\"older path $m' \colon I \to L(U,\bar{W})$ such that
\begin{align*}
m_{s,t} = m'_s X_{s,t} + m_{s,t}^{\#}
\end{align*}
for all $s,t \in I$ where $m^{\#} \colon I \times I \to \bar{W}$ satisfies $\| m^{\#} \|_{2\gamma ; I} < \infty$.
The path $m'$ is called a \emph{Gubinelli derivative} of $m$. We use $\mathscr{D}_{X}^{\gamma}(I,\bar{W})$ to denote the space of controlled paths based on $X$ on the interval $I$. We will sometimes just write $\mathscr{D}_{X}^{\gamma}(I)$ or $\mathscr{D}_{X}^{\gamma}$ if codomain or domain are clear from the context. It can be shown that this space is a Banach space with norm 
\begin{align*}
 \|m\|_{\mathscr{D}_{X}^{\gamma}} := \|(m,m')\|_{\mathscr{D}_{X}^{\gamma}} := |m_a| + |m'_a| + \|m'\|_{\gamma;I} + \| m^{\#} \|_{2\gamma; I}.
 \end{align*}
 If $\alpha \leq \beta \leq \gamma$, the space $\mathscr{D}_{X}^{\alpha,\beta}(I,\bar{W})$ is defined as the closure of  $\mathscr{D}_{X}^{\beta}(I,\bar{W})$ in the space  $\mathscr{D}_{X}^{\alpha}(I,\bar{W})$. It can be shown that $\mathscr{D}_{X}^{\alpha,\beta}(I,\bar{W})$ is separable for $\alpha < \beta$ \cite[Lemma 3.9]{GRS19}. 
 
 A path $m \colon I \to \bar{W}$ is a \emph{delayed controlled path} based on $X$ on the interval $I$ if there exist $\gamma$-H\"older paths $\zeta^0, \zeta^1 \colon I \to L(U,\bar{W})$ such that
\begin{align*}
m_{s,t} = \zeta^0_s X_{s,t} + \zeta^1_s X_{s-r,t-r} + m^{\#}_{s,t}
\end{align*}
for all $s,t \in I$ where $m^{\#} \colon I \times I \to \bar{W}$ satisfies $\| m^{\#} \|_{2\gamma ; I} < \infty$.
We use $\mathcal{D}_{X}^\gamma(I,\bar{W})$ to denote the space of delayed controlled paths based on $X$ on the interval $I$. A norm on this space can be defined by
\begin{align*}
 \| m \|_{\mathcal{D}_{X}^{\gamma}} := \|(m,\zeta^0,\zeta^1)\|_{\mathcal{D}_{X}^{\gamma}} := |m_a| + |\zeta^0_a| + |\zeta^1_a| + \|\zeta^0\|_{\gamma;I} + \|\zeta^1\|_{\gamma;I} + \|m^{\#}\|_{2\gamma; I}.
\end{align*}

\end{itemize}

We recall the concept of a random dynamical system introduced by L.~ Arnold \cite{Arn98}. 

\begin{itemize}
 \item Let $(\Omega,\mathcal{F})$ and $(X,\mathcal{B})$ be measurable spaces. Let $\mathbb{T}$ be either $\R$ or $\Z$, equipped with a $\sigma$-algebra $\mathcal{I}$ given by the Borel $\sigma$-algebra $\mathcal{B}(\R)$ in the case of $\mathbb{T} = \R$ and by $\mathcal{P}(\Z)$ in the case of $\mathbb{T} = \Z$.
A family $\theta=(\theta_t)_{t \in \mathbb{T}}$ of maps from $\Omega$ to itself is called a \emph{measurable dynamical system} if
\begin{itemize}
   \item[(i)] $(\omega,t) \mapsto \theta_t\omega$ is $\mathcal{F} \otimes \mathcal{I} / \mathcal{F}$-measurable,   \vspace{0.07cm}

   \item[(ii)] $\theta_0 = \operatorname{Id}$,   \vspace{0.07cm}
  
   \item[(iii)] $\theta_{s + t} = \theta_s \circ \theta_t$, for all $s,t \in \mathbb{T}$.   \vspace{0.1cm}
\end{itemize}
If $\mathbb{T} = \mathbb{Z}$, we will also use the notation $\theta := \theta_1$, $\theta^n := \theta_n$ and $\theta^{-n} := \theta_{-n}$ for $n \geq 1$.
If $\P$ is furthermore a probability on $(\Omega,\mathcal{F})$ that is invariant under any of the elements of $\theta$,
$$
\P \circ \theta_t^{-1} = \P
$$ 
for every $t \in \mathbb{T}$, we call the tuple $\big(\Omega, \mathcal{F},\P,\theta\big)$ a \emph{measurable metric dynamical system}. The system is called \emph{ergodic} if every $\theta$-invariant set has probability $0$ or $1$.

\item Let $\mathbb{T}^+ := \{t \in \mathbb{T}\, :\, t \geq 0\}$, equipped with the trace $\sigma$-algebra. An  \emph{(ergodic) measurable random dynamical system} on $(X,\mathcal{B})$ is an (ergodic) measurable metric dynamical system $\big(\Omega, \mathcal{F},\P,\theta\big)$ with a measurable map 
   $$
   \varphi \colon \mathbb{T}^+ \times \Omega \times X \to X
   $$ 
   that enjoys the \emph{cocycle property}, i.e. $\varphi(0,\omega,\cdot) = \operatorname{Id}_X$, for all $\omega \in \Omega$, and
  \begin{align*}
   \varphi(t+s,\omega,\cdot) = \varphi(t,\theta_s\omega,\cdot) \circ \varphi(s,\omega,\cdot)
  \end{align*}
  for all $s,t \in \mathbb{T}^+$ and $\omega \in \Omega$. The map $\varphi$ is called \emph{cocycle}. If $X$ is a topological space with $\mathcal{B}$ being the Borel $\sigma$-algebra and the map $\varphi(\cdot, \omega,\cdot) \colon \mathbb{T}^+ \times X \to X$ is continuous for every $\omega \in \Omega$, it is called a \emph{continuous (ergodic) random dynamical system}. In general, we say that \emph{$\varphi$ has property $P$} if and only if $\varphi(t,\omega,\cdot) \colon X \to X$ has property $P$ for every $t \in \mathbb{T}^+$ and $\omega \in \Omega$ whenever the latter statement makes sense.
\end{itemize}

We finally define measurable fields of Banach spaces and cocycles acting on it.

\begin{itemize}
 \item Let $(\Omega,\mathcal{F})$ be a measurable space. A family of Banach spaces $\{E_{\omega}\}_{\omega \in \Omega}$ is called a \emph{measurable field of Banach spaces} if there is a set of sections
 \begin{align*}
  \Delta \subset \prod_{\omega \in \Omega} E_{\omega}
 \end{align*}
 with the following properties:
 \begin{itemize}
  \item[(i)] $\Delta$ is a linear subspace of $\prod_{\omega \in \Omega} E_{\omega}$.
  \item[(ii)] There is a countable subset $\Delta_0 \subset \Delta$ such that for every $\omega \in \Omega$, the set $\{g(\omega)\, :\, g \in \Delta_0\}$ is dense in $E_{\omega}$.
  \item[(iii)] For every $g \in \Delta$, the map $\omega \mapsto \| g(\omega) \|_{E_{\omega}}$ is measurable.
 \end{itemize}

 \item Let $(\Omega,\mathcal{F},\P,\theta)$ be a measurable metric dynamical system and $(\{E_{\omega}\}_{\omega \in \Omega},\Delta)$ a measurable field of Banach spaces. A \emph{continuous cocycle on $\{E_{\omega}\}_{\omega \in \Omega}$} consists of a family of continuous maps
 \begin{align}\label{eqn:def_cocycle}
  \varphi(\omega, \cdot) \colon E_{\omega} \to E_{\theta \omega}.
 \end{align}
 If $\varphi$ is a continuous cocycle, we define $\varphi(n,\omega,\cdot) \colon E_{\omega} \to E_{\theta^n \omega}$ as 
 \begin{align*}
  \varphi(n,\omega,\cdot) := \varphi(\theta^{n-1}\omega,\cdot) \circ \cdots \circ \varphi(\omega,\cdot).
 \end{align*}
 We say that \emph{$\varphi$ acts on $\{E_{\omega}\}_{\omega \in \Omega}$} if the maps
 \begin{align}\label{eqn:measurability_cocycle}
  \omega \mapsto \| \varphi(n,\omega,g(\omega)) \|_{E_{\theta^n \omega}}, \quad n \in \N
 \end{align}
 are measurable for every $g \in \Delta$. In this case, we will speak of a \emph{continuous random dynamical system on a field of Banach spaces}. If the map \eqref{eqn:def_cocycle} is bounded linear/compact/differentiable, we call $\varphi$ a bounded linear/compact/differentiable cocycle.
\end{itemize}


\section{Properties of nonlinear rough delay equations}\label{sec:sdde_nonlinear}

In this section, we study different aspects of nonlinear rough delay differential equations. For simplicity, we will study equations without a drift coefficient first. Fix a delay $r > 0$ and consider
 \begin{align}\label{eqn:rough_delay}
  \begin{split}
  y_t &= \xi_0 + \int_0^t \sigma (y_s, y_{s-r})\, d \mathbf{X}_s; \quad t \in [0,r] \\
  y_t &= \xi_{t}; \quad t \in [-r, 0] 
  \end{split}
\end{align}
where $\mathbf{X} = (X,\mathbb{X},\mathbb{X}(-r))$ is a delayed $\gamma$-rough path, $\gamma \in (1/3,1/2]$, and $X \colon \R \to U$ is locally $\gamma$-H\"older continuous. We recall the following result:
\begin{theorem}\label{thm:delay_existence}
 Assume $\sigma \in C^3_b(W^2,L(U,W))$, $1/3 < \alpha \leq \beta < \gamma \leq 1/2$ and either $\xi \in \mathscr{D}_{X}^{\beta}([-r,0],W)$ or $\xi \in \mathscr{D}_{X}^{\alpha, \beta}([-r,0],W)$. Then the equation \eqref{eqn:rough_delay} has a unique solution $y \in \mathscr{D}_{X}^{\beta}([0,T],W)$ resp. $y \in \mathscr{D}_{X}^{\alpha, \beta}([0,T],W)$ for any $T > 0$. In both cases, $y_t' = \sigma(y_t,y_{t-r})$.
\end{theorem}

\begin{proof}
 The case  $\xi \in \mathscr{D}_{X}^{\beta}([-r,0],W)$ was shown in \cite[Theorem 1.8]{GRS19} and the case $\xi \in \mathscr{D}_{X}^{\alpha,\beta}([-r,0],W)$ follows from continuity of the solution map, cf. \cite[Theorem 1.9]{GRS19}.
\end{proof}

\subsection{Regularity}

In this subsection, we will study the regularity of the solution map induced by \eqref{eqn:rough_delay}.
More precisely, we will give sufficient conditions under which this map is differentiable in the initial condition, which means differentiability in Fr\'echet-sense on the space of controlled paths. To prove our result, we will follow a similar strategy as in \cite{Bai15b} and \cite{CL18}.\\ 

\begin{definition}
 For $ m\in\mathbb{N} $ and $0 <\kappa \leqslant1$, we say that $ f \colon V^2 \to W$ belongs to $ \mathscr{C}^{m+\kappa}(V^{2},W) $ if its derivatives up to order $ m $ are bounded and continuous and if $ D^{m}f $ is $\kappa$- H\"older continuous. The space is equipped by the norm
\begin{align*}
\Vert f\Vert _{\mathscr{C}^{m+\kappa}}=\max_{j=0,...,m}\lbrace \Vert D^{j}f\Vert_{\infty},\Vert D^{m}f\Vert _{\kappa}\rbrace.
\end{align*}
\end{definition}
Next, we give a more general definition of a delayed controlled path.

\begin{definition}
Let $I = [a,b]$. We say that $m \colon I \rightarrow W $ is a \emph{delayed $ (\alpha,\beta,\theta)$-controlled path based on $X$ on the interval $I$} if there exist paths $\zeta^0, \zeta^1 \colon I \to L(U,\bar{W})$ such that
\begin{align*}
m_{s,t} = \zeta^0_s X_{s,t} + \zeta^1_s X_{s-r,t-r} + m^{\#}_{s,t}
\end{align*}
holds for all $s,t \in I$ where
\begin{align*}
\| m \|_{\alpha ; I} , \| \zeta^{0} \|_{\beta; I},  \| \zeta^{1} \|_{\beta; I}\ \text{and} \ \| m^{\#} \|_{\theta ; I} < \infty.
\end{align*}
We denote the corresponding space by $ \mathcal{D}_X^{\alpha, \beta, \theta}(I,\bar{W}) $ where the norm on this space is defined as
\begin{align}\label{eqn:norm_delayed_controlled_path}
 \| m \|_{\mathcal{D}_{X}^{\gamma}} := \|(m,\zeta^0,\zeta^1)\|_{\mathcal{D}_{X}^{\gamma}} := |m_a| + |\zeta^0_a| + |\zeta^1_a| + \| m \|_{\alpha;I} + \|\zeta^0\|_{\beta;I} + \|\zeta^1\|_{\beta;I} + \|m^{\#}\|_{\theta; I}.
\end{align}
\end{definition}

\begin{remark}\label{sewing 2}
  Clearly, $ \mathcal{D}_X^{\beta, \beta, 2 \beta}(I,\bar{W})  = \mathcal{D}_X^{\beta}(I,\bar{W}) $. Using the sewing lemma \cite[Lemma 4.2]{FH14}, it is easy to check that we can define an integral of the form
  \begin{align*}
   \int m\, d\mathbf{X}
  \end{align*}
  as in \cite[Theorem 1.5]{GRS19} for delayed $\gamma$-rough paths $\mathbf{X}$ and delayed $(\alpha,\beta,\theta)$-controlled paths $m$ provided $ \theta+\gamma>1 $ and $ \beta+2\gamma>1 $. Furthermore, the (linear) map
\begin{align*}
 \mathcal{D}_X^{\alpha, \beta, \theta}(I, L(U,W))  &\rightarrow  \mathcal{D}_X^{\gamma, \alpha, 2 \gamma}(I,W) \\
m &\mapsto \int m\, d\mathbf{X}
\end{align*}
is well defined and continuous .
\end{remark}

The next theorem is a version of the Omega lemma \cite[Proposition 5]{CL18} for delayed controlled paths.

\begin{theorem}(Delayed Omega lemma)\label{omega}
Let $ n\in\mathbb{N}$ and $ 0<\kappa\leqslant 1 $ for $G \in \mathscr{C}^{n+1+\kappa}(V^{2} ,W)$, $ \eta\in (0,1)$ and $r > 0$. Then the map
\begin{align*}
\mathfrak{D}G : \mathscr{D}_{X}^\beta([0,r],V)\times \mathscr{D}_{X}^\beta([-r,0],V) &\rightarrow \mathcal{D}_X^{\beta,\beta\eta\kappa,\beta(1+\eta\kappa)\wedge 2\beta}([0,r],W) \\
\big{(}y_{t},\xi_{t-r}\big{)}_{t\in[0,r]} &\mapsto \big{(}G (\xi_{0}+y_{t},\xi_{t-r})\big{)}_{t\in[0,r]}
\end{align*}
is locally of class $ \mathscr{C}^{n+\kappa(1-\eta)} $. 
\end{theorem}

\begin{proof}
We noted in \cite[Remark 1.4]{GRS19} that every delayed controlled path based on $X$ can be seen as a usual controlled path based on $(X,X_{\cdot - r})$ and vice versa. Using this identification, the assertion just follows from  \cite[Proposition 5]{CL18}.
\end{proof}

Thanks to the delayed Omega lemma, we can state the following theorem:

\begin{theorem}\label{thm:reg_sol_map}
Let $0<\kappa\leqslant 1$, $ 2\leqslant n+\kappa$ and $\sigma \in\mathscr{C}^{n+1+\kappa}(W^{2},L(U,W))$. For a delayed $\gamma$-rough path $\mathbf{X}$, consider equation \eqref{eqn:rough_delay}. Then, under the same assumptions as in Theorem \ref{thm:delay_existence}, the solution map induced by \eqref{eqn:rough_delay} is locally of class $ \mathscr{C}^{n+\kappa(1-\eta)} $ for any $ \eta\in (0,1) $ provided $ \beta\big{(}2+\kappa\eta\big{)}> 1 $.
\end{theorem}

\begin{proof}
Fix $\hat{\xi} \in \mathscr{D}_{X}^{\beta}([-r,0],W)$. We aim to prove the claimed regularity in a neighbourhood around $\hat{\xi}$. Choose  $ M>0 $ such that 
\begin{align*}
 \hat{\xi} \in  B := \big{\lbrace} {\xi} \in\mathscr{D}_{X}^{\beta}([-r,0],W) , \ \Vert{\xi}\Vert_{\mathscr{D}_{X}^{\beta}([-r,0],W)}<M\big{\rbrace}.
\end{align*}
 Let $\mathscr{D}_{X,0}^\beta([a,b],W)$ be the set of functions in $ \mathscr{D}_{X}^\beta([a,b],W) $ starting from $ 0 $. Let $0 <  t_{0}\leqslant r $ and define
\begin{align}\label{map}
\Gamma : B \times \mathscr{D}_{X,0}^{\beta}([0,t_{0}],W) &\to \mathscr{D}_{X,0}^\beta([0,t_{0}],W)\nonumber \\ 
\big{(}\xi_{t-r},y_{t}\big{)}_{0\leqslant t\leqslant t_{0}} &\mapsto \bigg{(}\int_{0}^{t} \sigma(y_{\tau}+\xi_{0},\xi_{\tau -r})d\mathbf{X}_{\tau}\bigg{)}_{0\leqslant t\leqslant t_{0}}.
\end{align} 
Note that by Remark \ref{sewing 2} and Theorem \ref{omega}, this map is locally of class $ \mathscr{C}^{n+\kappa(1-\eta)} $. Using the estimates (59) and (61) in \cite{NNT08}, we see that
\begin{align}\label{estimatee}
\begin{split}
\Vert\Gamma(\xi ,y)\Vert_{\mathscr{D}_{X}^{\beta}[0,t_{0}]}&\leqslant C_{1}A^{3}\big{(}1+\Vert\xi\Vert_{\mathscr{D}_{X}^{\beta}[-r,0]}^{2}\big{)}\big{(}1 + t_0^{\gamma -\beta}\Vert y\Vert ^{2}_{\mathscr{D}_{X}^{\beta}[0,t_{0}]}\big{)} \qquad \\
\Vert\Gamma(\xi ,y)-\Gamma(\xi ,\tilde{y}) \Vert_{\mathscr{D}_{X}^{\beta}[0,t_{0}]}&\leqslant C_{1}A^{3 }\big{(}1+\Vert y\Vert_{\mathscr{D}_{X}^{\beta}[0,t_{0}]}+\Vert \tilde{y}\Vert_{\mathscr{D}_{X}^{\beta}[0,t_{0}]}+\Vert\xi\Vert_{\mathscr{D}_{X}^{\beta}[-r,0]}\big{)}^{2}\Vert y-\tilde{y}\Vert_{\mathscr{D}_{X}^{\beta}[0,t_{0}]}t_0^{\gamma -\beta}
\end{split}
\end{align}
where $ C_{1}$ only depends on $\sigma$. Let
$ C := C_{1}A^{3}(1+M^{2}) $ and set $ \tau_{1} := (8C^{2})^{\frac{-1}{\gamma-\beta}} $. From \cite[Lemma 4.1]{NNT08},
\begin{align}\label{ess}
\sup\big{\lbrace} u\in\mathbb{R}^{+}\, :\, C(1 + \tau_1^{\gamma-\beta} u^2 )\leqslant u \big{\rbrace}\leqslant (4+2\sqrt{2})C =: M_{1}.
\end{align}
Choose $ \tau_{2}$ such that
\begin{align*}
C_{1}A^{3}(1+2M_{1}+M)^2 \tau_{2}^{\gamma-\beta} \leq \frac{1}{2}.
\end{align*}
Set $ \tau_{3} := \min\lbrace\tau_{1},\tau_{2},r\rbrace $. Choosing $\tau_3$ smaller if necessary, we can assume that $ N := \frac{r}{\tau_{3}}\in\mathbb{N} $. Set
\begin{align*}
B_{1} := \bigg{\lbrace}y\in\mathscr{D}_{X,0}^{\beta}([0,\tau_3],W)\, :\, \Vert y\Vert_{\mathscr{D}_{X,0}^{\beta}([0,\tau_3],W)}\leqslant M_{1}\bigg{\rbrace}.
\end{align*}
With this choice, the map 
\begin{align*}
 \Gamma_1 := \Gamma \vert_{B \times B_1} \colon B \times B_1 \to  B_1
\end{align*}
is well defined. Moreover, for fixed $\hat{\xi} \in B$,
\begin{align*}
\Lambda_{1} : B_{1} &\to B_{1}\\
(y_{s})_{ 0\leqslant s\leqslant\tau_{3}} &\mapsto \bigg{(}\int_{0}^{s} \sigma(\hat{\xi}_{0}+y_{\tau},\hat{\xi}_{\tau-r})\, d\mathbf{X}_{\tau}\bigg{)}_{0\leqslant s\leqslant\tau_{3}}
\end{align*} 
is a contraction, so it admits a unique fixed point which we denote by $ (z^{1,\hat{\xi}}_{s})_{0\leqslant s\leqslant \tau_{3}} $. This shows that we can use the implicit function theorem on Banach spaces (cf. \cite[2.5.7 Implicit Function Theorem]{AMR88} or \cite[Theorem 1]{CL18}) to see that there is a neighbourhood $U$ around $\hat{\xi}$ such that for every $\xi \in U$, there are functions  $ (z^{1,{\xi}}_{s})_{0\leqslant s\leqslant \tau_{3}} $ with the property that $\Lambda_1(z^{1,\xi}) = z^{1,\xi}$ and the map $\xi \mapsto z^{1,\xi}$ is of class $ \mathscr{C}^{n+\kappa(1-\eta)} $. Therefore, $\xi \mapsto (y^{1,\xi}_{s}=\xi_{0}+z^{1,\xi}_{s})_{0\leqslant s\leqslant \tau_{3}} $, which is the solution of equation \eqref{eqn:rough_delay} in $ [0,\tau_{3}] $, is also locally of class $ \mathscr{C}^{n+(1-\eta)\kappa} $. Moreover,
\begin{align}
\Vert z^{1,\xi}\Vert_{\mathscr{D}_{X}^\beta([0,\tau_{3}])} \leqslant (4+2\sqrt{2})C
\end{align}
holds for every $\xi \in U$. Now we proceed inductively. For $ 2\leqslant j\leqslant N $, define 
\begin{align*}
B_{j}=\bigg{\lbrace}y\in\mathscr{D}_{X,0}^{\beta}([(j-1)\tau_{3},j\tau_{3}],W) \, :\,  \Vert y\Vert_{\mathscr{D}_{X,0}^{\beta}[(j-1)\tau_{3},j\tau_{3}]}\leqslant M_{1}\bigg{\rbrace}
\end{align*}
and
\begin{align*}
\Lambda_{j} : B_{j} &\to B_{j}\\
\big{(}y_{s}\big{)}_{ (j-1)\tau_{3}\leqslant s\leqslant j\tau_{3}} &\mapsto \bigg{(}\int_{(j-1)\tau_{3}}^{s} \sigma (y^{j-1,\hat{\xi}}_{(k-1)\tau_{3}} + y_{\tau},\hat{\xi}_{\tau-r})d\mathbf{X}_{\tau}\bigg{)}_{(j-1)\tau_{3}\leqslant s\leqslant j\tau_{3}}.
\end{align*}
Again, this map is contraction and admits a unique fixed point, namely $\big{(}z^{j,\hat{\xi}}_{s}\big{)}_{(j-1)\tau_{3}\leqslant s\leqslant j\tau_{3}}$, and a locally defined map $\xi \mapsto \big{(}z^{j,{\xi}}_{s}\big{)}_{(j-1)\tau_{3}\leqslant s\leqslant j\tau_{3}}$ which is of class $ \mathscr{C}^{n+\kappa(1-\eta)} $. Again, 
\begin{align}\label{ZZS}
\Vert z^{j,\xi}\Vert_{\mathscr{D}_{X}^\beta([(j-1)\tau_{3},j\tau_{3}])} \leqslant (4+2\sqrt{2})C
\end{align}
holds for all $\xi$ in a neighbourhood around $\hat{\xi}$. This shows that $ (y^{j,\xi}_{s} = y^{j-1,\xi}_{(j-1)\tau_{3}}+z^{j,\xi}_{s})_{(j-1)\tau_{3}\leqslant s\leqslant j\tau_{3}} $, the solution of \eqref{eqn:rough_delay} in $ [(j-1)\tau_{3},j\tau_{3}] $, has the same local regularity. Finally, the following map is locally of class $ \mathscr{C}^{n+\kappa(1-\eta)} $:
\begin{align*}
\Lambda :B &\to \prod_{1\leqslant j\leqslant N}\mathscr{D}_{X}^{\beta}[(j-1)\tau_{3},j\tau_{3}]\\
\xi &\mapsto \prod_{1\leqslant j\leqslant N}\big{(}y^{j,\xi}_{s}\big{)}_{(j-1)\tau_{3}\leqslant s\leqslant j\tau_{3}}.
\end{align*}
Since we can consider $ \mathscr{D}_{X}^{\beta}[0,r] $ as a closed subspace of $\prod_{1\leqslant j \leqslant N}\mathscr{D}_{X}^{\beta}[(j-1)\tau_{3},j\tau_{3}] $, the regularity claim is proved. 
\end{proof}

\begin{remark}
Since $ C^{3}_{b}\subset \mathscr{C}^{3}$, Theorem \ref{thm:reg_sol_map} implies that the solution of \eqref{eqn:rough_delay} is Fr\'echet differentiable in the initial condition.
\end{remark}

The proof of Theorem \ref{thm:reg_sol_map} also reveals a bound for the solution to \eqref{eqn:rough_delay} which we record in the next theorem.

\begin{theorem}\label{thm:bound_rdde}
 Under the same assumptions as in Theorem \ref{thm:delay_existence}, there exists a polynomial $P:\mathbb{R}\times\mathbb{R}\rightarrow\mathbb{R} $ such that its coefficients depend on $ \sigma $, $\beta$ and $\gamma$ and if $y^{\xi}$ denotes the solution to \eqref{eqn:rough_delay} with initial condition $\xi$, we have
 \begin{align}\label{first-estimate}
\Vert y^{\xi}\Vert _{{\mathscr{D}_{X}^{\beta}([0,r])}}\leqslant P\big{(}A,\Vert\xi\Vert_{{\mathscr{D}_{X}^{\beta}([-r,0])}}\big{)}
\end{align}
 where $ A = 1 + \Vert \mathbf{X}\Vert_{\gamma,[0,r]}$.
\end{theorem}

\begin{proof}
 With the same notation as in the proof of Theorem \ref{thm:reg_sol_map},
\begin{align}\label{WZYXW}
\Vert (y^{\xi})^{\#}\Vert_{2\beta,[0,r]}\leqslant \sum_{1\leqslant k\leqslant N}\Vert(z^{k,\xi})^{\#}\Vert_{2\beta,[(k-1)\tau_{3},k\tau_{3}]} + r^{\gamma-\beta}\Vert X\Vert_{\gamma,[0,r]}\sum_{1\leqslant k\leqslant N}\Vert (z^{k,\xi})^{\prime}\Vert_{\beta,[(k-1)\tau_{3},k\tau_{3}]}.
\end{align}
The estimate \eqref{first-estimate} now follows from \eqref{ZZS}, \eqref{WZYXW}, subadditivity of the H\"older norm and our choice for $ \tau_{3} $.
\end{proof}

It is possible to show that all derivatives solve linear, non-autonomous rough delay equations obtained by formally taking the derivatives of \eqref{eqn:rough_delay}. We give a proof of this result for the first derivative in the next proposition. Higher order derivatives can be treated similarly.

\begin{proposition}\label{prop:linearization_eq}
For $\xi \in \mathscr{D}_X^{\beta}([-r,0],W)$, let $ (y^{\xi}_{t})_{0 \leq t \leq r} $ be the solution to \eqref{eqn:rough_delay}. The derivative of the solution at $\xi$ in the direction of $ \tilde{\xi} $ exists and satisfies the following equation:
\begin{align}{\label{MKL}}
  \begin{split}
Dy^{\xi}[\tilde{\xi}](t) - \tilde{\xi}_0 &= \int_{0}^{t}\big{[} \sigma_{x}(y^{\xi}_{\tau},\xi_{\tau -r})Dy^{\xi}[\tilde{\xi}](\tau) + \sigma_{y}(y^{\xi}_{\tau},\xi_{\tau-r})\tilde{\xi}_{\tau -r}\big{]}d\mathbf{X}_{\tau};\quad t \in [0,r] \\
Dy^{\xi}[\tilde{\xi}](t) &= \tilde{\xi}_t; \quad t \in [-r,0].
  \end{split}
\end{align}
\end{proposition}

\begin{proof}
By definition,
\begin{align*}
&\frac{y^{\xi+z\tilde{\xi}}_{s,t}-y^{\xi}_{s,t}}{z} - \int_{s}^{t} \big{[}\sigma_{x} (y^{\xi}_{\tau},\xi_{\tau -r})Dy^{\xi}[\tilde{\xi}](\tau) + \sigma_{y}(y^{\xi}_{\tau},\tau_{\tau-r})\tilde{\xi}_{\tau -r}\big{]}d\mathbf{X}_{\tau}\\
&=\int_{s}^{t}\bigg{[}\frac{ \sigma(y_{\tau}^{\xi+z\tilde{\xi}},\xi_{\tau}+z\tilde{\xi}_{\tau -r}) - \sigma(y^{\xi}_{\tau},\xi_{\tau-r})}{z}-\big{[} \sigma_{x}(y^{\xi}_{\tau},\xi_{\tau -r})Dy^{\xi}[\tilde{\xi}](\tau) + \sigma_{y}(y^{\xi}_{\tau},\xi_{\tau-r})\tilde{\xi}_{\tau -r}\big{]}\bigg{]}d\mathbf{X}_{\tau}\\
&=\int_{s}^{t}\bigg{[}\big{[}A^{z}_{\tau}M^{z}_{\tau} + B^{z}_{\tau}\big{]}-\big{[}A_{\tau}M_{\tau} + B_{\tau}\big{]}\bigg{]}d\mathbf{X}_{\tau}
\end{align*}
where 
\begin{align*}
A_{\tau}^{z} = \int_{0}^{1}\sigma_{x}\big{(}&\eta y_{\tau}^{\xi + z\tilde{\xi}}+(1-\eta)y_{\tau}^{z},\xi_{\tau -r},+\eta z\tilde{\xi}_{\tau -r}\big{)}d\eta\ \ \ ,\ \ \ M_{\tau}^{z}=\frac{y_{\tau}^{\xi+z\tilde{\xi}}-y_{\tau}^{\xi}}{z}\\
& B_{\tau}^{z}=\int_{0}^{1} \sigma_{y}\big{(}\eta y_{\tau}^{\xi+z\tilde{\xi}}+(1-\eta)y_{\tau}^{\xi},\xi_{\tau -r},+\eta z\tilde{\xi}_{\tau -r}\big{)}\tilde{\xi}_{\tau -r}d\eta 
\end{align*}
and
\begin{align*}
 A_{\tau} = \sigma_{x}(y^{\xi}_{\tau},\xi_{\tau -r}),\quad M_{\tau} = Dy^{\xi}[\tilde{\xi}](\tau),\quad B_{\tau} = \sigma_{y}(y^{\xi}_{\tau},\xi_{\tau-r})\tilde{\xi}_{\tau -r}.
\end{align*}
Note that by Theorem (\ref{thm:reg_sol_map}), $ \lim_{z\rightarrow 0}\Vert M_{.}^{z}-M_{.}\Vert_{\mathscr{D}_{X}^{\beta}[0,r]} =0 $. From continuity in the initial condition, we furthermore see that $ \lim_{z\rightarrow 0}\Vert y^{\xi +z\tilde{\xi}}-y^{\xi}\Vert_{\mathscr{D}_{X}^{\beta}[0,r]} =0$. Consequently, thanks to our assumptions on $ \sigma$, it is not hard too see that
\begin{align*}
\lim_{z\rightarrow 0}\bigg{[}\big{\Vert} [A^{z}_{.}M^{z}_{.}+B^{z}_{.}]-[A_{.}M_{.}+B_{.}]\big{\Vert} _{\ \mathcal{D}_{X}^{\beta}[0,r]}\bigg{]}=0.
\end{align*} 
Using remark (\ref{sewing 2}), equality (\ref{MKL}) can be verified.
\end{proof}

\subsection{Rough delay equations with a linear drift}

Our next goal is to generalize the theory in order to include a drift term in the equation. More precisely, we aim to solve the equation
 \begin{align}\label{drift}
\begin{split}
&dy_{t}=B(y_t,y_{t-r})dt + \sigma (y_{t},y_{t-r})d \mathbf{X}_{t}\\
&y_{s}=\xi_{s}, \ \ \ -r\leqslant s\leqslant 0
\end{split}
\end{align}
with initial condition $\xi \in \mathscr{D}_{X}^\beta([-r , 0],W)$ for a linear drift $ B:W^{2}\rightarrow W $ and to give a bound for the solution map. We believe that we could even include a nonlinear drift satisfying suitable growth assumptions as in \cite{RS17}, but we restrict ourselves to a linear drift here for the sake of simplicity. The next theorem is the main result of this section.

\begin{theorem}\label{thm:SDDE_linear_drift}
 Let $\sigma \in C^4_b$. Then the equation \eqref{drift} has a unique solution $y \in \mathscr{D}_X^{\beta}([0,r],W)$. Moreover, there is a polynomial $Q$ depending on $B$, $\sigma$, $\gamma$ and $\beta$ such that
 \begin{align*}
  \| y \|_{\mathscr{D}_X^{\beta}([0,r])} \leq Q (A,\Vert\xi\Vert_{{\mathscr{D}_{X}^{\beta}([-r,0])}})
 \end{align*}
where $ A = 1+\Vert \mathbf{X} \Vert_{\gamma,[0,r]} $.
\end{theorem}

\begin{proof}
 The idea is to give a representation of the solution to \eqref{drift} using the flow map of the respective equation omitting the drift term. Let $ \xi\in\mathscr{D}_{X}^\beta([-r,0],W) $ be fixed and consider the equation
\begin{align}\label{eqn:reduced_without_drift}
\begin{split}
&dy_{t} = \sigma(y_{t},\xi_{t-r})\, d \mathbf{X}_{t}\\
&y_{s}=x, \ \ \ 0\leqslant s\leqslant t\leqslant r.
\end{split}
\end{align}
Existence and uniqueness of this equation can be shown similarly to the usual delay case. We use $ \bar{\varphi}(s,t,x) $ to denote the solution of \eqref{eqn:reduced_without_drift} at time $t$ with initial condition $y_s = x$. From uniqueness of the solution, we have for every $ \tau\leqslant s\leqslant t $,
\begin{align*}
\bar{\varphi}(\tau,t,x)=\bar{\varphi}\big{(}s,t,\bar{\varphi}(\tau,s,x)\big{)}.
\end{align*}
As for usual rough differential equations \cite[Theorem 10.14]{FV10}, one can show that there is a polynomial $ P_1 $ such that
\begin{align}\label{NMK}
\sup_{x\in W, 0\leqslant s\leqslant t\leqslant r}\Vert\bar{\varphi}(s,t,x)-x\Vert\leqslant (t-s)^{\beta}P_1{(}A,\Vert\xi\Vert_{{\mathscr{D}_{X}^{\beta}([-r,0])}}{)}.
\end{align}
In addition, one can check that the solution is differentiable with respect to initial value and that its derivative is the matrix solution of the equation
\begin{align*}
D\bar{\varphi}(s,t,x)-I = \int_{s}^{t} \sigma_x(\bar{\varphi}(s,\tau,x),\xi_{\tau-r})D\bar{\varphi}(s,\tau,x)d \mathbf{X}_{\tau}.
\end{align*}
Let $ 0<t_{0}<r$ be fixed. For $0\leqslant \tau < \varsigma \leqslant t_{0}$, we define
\begin{align*}
\tilde{X}_{\tau}:=X_{t_{0}-\tau} ,\ \ \   \tilde{\mathbb{X}}_{\tau,\varsigma}:=-\mathbb{X}_{t_{0}-\varsigma , t_{0}-\tau},\ \ \ \tilde{\mathbb{X}}_{\tau,\varsigma}(-r):=-\mathbb{X}_{t_{0}-\varsigma , t_{0}-\tau}(-r).
\end{align*}
We say that $ \eta\in \tilde{\mathscr{D}}_{\tilde{X}}^{\beta}([a,b],W) $ if we have a decomposition of the form 
\begin{align*}
\eta_{s,t} = \eta^{\prime}_{t}\tilde{X}_{s,t}+\eta^{\#}_{s,t}
\end{align*}
where 
\begin{align*}
\Vert\eta^{\prime}\Vert_{\beta;[a,b]} < \infty \quad \text{and} \quad \sup_{s<t}\frac{\vert\eta^{\#}_{s,t}\vert}{(t-s)^{2\beta}}<\infty.
\end{align*}
Using the sewing lemma \cite[Lemma 4.2]{FH14} we can also define 
\begin{align*}
\int_{[a,b]}\eta_{\tau}d\tilde{\mathbf{X}}_{\tau} &:= \lim_{\vert\Pi\vert\rightarrow 0}\sum_{\Pi}\big{[}\eta_{\tau_{j+1}}\tilde{X}_{\tau_{j},\tau_{j+1}}+\eta^{\prime}_{\tau_{j+1}}\tilde{\mathbb{X}}_{\tau_{j},\tau_{j+1}}\big{]}\\
\int_{[a,b]}\eta_{\tau-r}d\tilde{\mathbf{X}}_{\tau} &:= \lim_{\vert\Pi\vert\rightarrow 0}\sum_{\Pi}\big{[}\eta_{\tau_{j+1}-r}\tilde{X}_{\tau_{j},\tau_{j+1}}+\eta^{\prime}_{\tau_{j+1}-r}\tilde{\mathbb{X}}_{\tau_{j},\tau_{j+1}}(-r)\big{]}.
\end{align*}
For $ \xi\in\mathscr{D}_{X}^\beta([a,b],W)$, it is straightforward to check that $ \tilde{\xi}_{\cdot} := \xi_{t_{0} - \cdot}\in\tilde{\mathscr{D}}_{\tilde{X}}^{\beta}([t_{0}-b,t_{0}-a],W)   $ and  that
\begin{align*}
\int_{[a,b]}\xi_{\tau}d\mathbf{X}_{\tau}=\int_{[t_{0}-b,t_{0}-a]}\tilde{\xi}_{\tau}d\tilde{\mathbf{X}}_{\tau}.
\end{align*}
For $ s_{0}\leqslant t_{0}\leqslant r $ and $ \tilde{\varphi}(s_{0},t,x) := \bar{\varphi}(s_{0},t_{0}-t,x) $ we consider the equation
\begin{align}\label{invers_}
\begin{split}
  dZ_{t} &= \sigma_x \big{(} \tilde{\varphi}(s_{0},t,x_{0}),\tilde{\xi}_{t-r}\big{)}Z_{t}d\tilde{\mathbf{X}}_{t}\\
  Z_{0} &= I,\ \ \ 0\leqslant t\leqslant t_{0}-s_{0}.
\end{split}
\end{align}
Then 
\begin{align*}
Z_{t_{0}-s_{0}}=[D\bar{\varphi}(s_{0},t_{0},x)]^{-1}.
\end{align*}
Thus by standard estimates for linear equations \cite[Theorem 10.53]{FV10}, we have a bound of the form
\begin{align}\label{in-b}
\sup_{s\leqslant t\leqslant r, x\in W} \Vert [D\bar{\varphi}(s,t,x)]^{-1}-I\Vert\leqslant M(t-s)^{\beta} P_2 \big{(}A,\Vert\xi\Vert_{{\mathscr{D}_{X}^{\beta}([-r,0])}}\big{)} \exp\big{(}(t-s) P_2 \big{(}A,\Vert\xi\Vert_{{\mathscr{D}_{X}^{\beta}([-r,0])}}\big{)}\big{)}
\end{align}
where $ M $ is just a general constant and $ P_2 $ is a polynomial. Now we consider the ODE
\begin{align*}
\begin{split}
  d \eta_{t} &= [D\bar{\varphi}(0,t,\eta_{t})]^{-1}B\big{(}\bar{\varphi}(0,t,\eta_{t}),\xi_{t-r}\big{)}dt\\
  \eta_{0} &= \xi_{0}.
\end{split}
\end{align*} 
Using the chain rule, it is straightforward to see that $ \bar{\varphi}(0,t,\eta_{t}) $ solves \eqref{drift}. Next, we choose $\tau > 0$ sufficiently small such that
\begin{align*}
M\tau^{\beta}P_2(A,\Vert\xi\Vert_{{\mathscr{D}_{X}^{\beta}([-r,0])}})\exp(\tau P_2(A,\Vert\xi\Vert_{{\mathscr{D}_{X}^{\beta}([-r,0])}}))\leqslant 1
\end{align*}
holds. Using some basic calculations, we can check that there is a polynomial $P_3$ such that
\begin{align}\label{akk}
\frac{r}{\tau} = P_3(A,\Vert\xi\Vert_{{\mathscr{D}_{X}^{\beta}([-r,0])}}).
\end{align}
Choosing $\tau$ smaller if necessary, we can assume that there is some $ n\in\mathbb{N} $ such that $ n\tau =r $. Define $I_{m} := [(m-1)\tau ,m\tau]$ for $1 \leq m\leqslant n $ and $ \eta^{0}_{0} := \xi_{0} $. Inductively, we define the equations
\begin{align}\label{DDec}
\begin{split}
  d\eta^{m}_{t} &= [D\bar{\varphi}_{x}((m-1)\tau,t,\eta^{m}_{t})]^{-1}B\big{(}\bar{\varphi}((m-1)\tau ,t,\eta^{m}_{t}),\xi_{t-r}\big{)}dt, \ \ t\in [(m-1)\tau ,m\tau] \\
  \eta^{m}_{(m-1)\tau} &= \bar{\varphi}((m-1)\tau ,\eta^{m-1}_{(m-1)\tau}).
\end{split}
\end{align} 
Again, it is not hard to see that
\begin{align*}
y_{t}=\bar{\varphi}((m-1)\tau, t, \eta^{m}_{t}), \  \ \ t\in [(m-1)\tau , m\tau]
\end{align*}
solves \eqref{drift}. From \eqref{in-b}, 
\begin{align*}
\Vert\eta^{m}_{t}\Vert-\Vert\eta^{m}_{(m-1)\tau}\Vert\leqslant 2\Vert B\Vert\int_{(m-1)\tau}^{t}\big{[}\Vert\bar{\varphi}((m-1)\tau, \varsigma,\eta^{m}_{\varsigma})\Vert + \Vert\xi_{\varsigma -r}\Vert\big{]} d\varsigma.
\end{align*}
By Gr\"onwall's lemma and \eqref{NMK}, we can deduce that there is for a constant $ M $ and polynomial $P_4 $ such that
\begin{align*}
\Vert\eta^{m}\Vert_{\infty;I_{m}}\leqslant\exp(2\Vert B\Vert\tau)\Vert\eta^{m}\Vert_{\infty;I_{m-1}} +M\big{[}\exp(2\Vert B\Vert\tau)-1\big{]}\big{[}\Vert\xi\Vert_{\infty}+P_4(A,\Vert\xi\Vert_{{\mathscr{D}_{X}^{\beta}([-r,0])}})\big{]}.
\end{align*}
Finally, from \eqref{NMK} and \eqref{akk}, for a polynomial $P_5$,
\begin{align}\label{GC}
\Vert y\Vert_{\infty;[0,r]}\leqslant P_5(A,\Vert\xi\Vert_{{\mathscr{D}_{X}^{\beta}([-r,0])}}).
\end{align}
Remember that
\begin{align*}
y_{s,t}=\int_{s}^{t} B(y_{\varsigma},\xi_{\varsigma -r})\, d\varsigma + \int_{s}^{t}\sigma(y_{\varsigma},\xi_{\varsigma -r})\, d \mathbf{X}_{\varsigma}.
\end{align*}
Using the standard estimate for the rough integral \cite[Theorem 4.10]{FH14} and \eqref{GC}, we obtain for $ 0\leqslant s<t\leqslant r $
\begin{align}\label{BGBN}
\Vert y\Vert_{\beta;[s,t]}+\Vert y^{\#}\Vert_{2\beta ;[s,t]}\leqslant P_6 (A,\Vert\xi\Vert_{{\mathscr{D}_{X}^{\beta}([-r,0])}})+(t-s)^{\gamma -\beta} P_7(A,\Vert\xi\Vert_{{\mathscr{D}_{X}^{\beta}([-r,0])}})[\Vert y\Vert_{\beta;[s,t]}+\Vert y^{\#}\Vert_{2\beta ;[s,t]}]
\end{align}
where $ P_6$ and $ P_7 $ are polynomials. Again, we can find a polynomial $ P_8 $ and $ \tau > 0$ such that
\begin{align*}
\frac{r}{\tau}=P_8(A,\Vert\xi\Vert_{{\mathscr{D}_{X}^{\beta}([-r,0])}}) \ \ \text{and} \ \ \ \tau^{\gamma -\beta}P_7(A,\Vert\xi\Vert_{{\mathscr{D}_{X}^{\beta}([-r,0])}})\leqslant \frac{1}{2}.
\end{align*}
Finally, from \eqref{BGBN} and subadditivity of the H\"older norm, we can deduce the existence of a polynomial $ Q $ such that
\begin{align}
\Vert y\Vert_{\mathscr{D}_{X}^{\beta}([0,r])}\leqslant  Q (A,\Vert\xi\Vert_{{\mathscr{D}_{X}^{\beta}([-r,0])}}).
\end{align}

\end{proof}

\begin{corollary}\label{corollary:results_linear_true}
 Under the same assumptions as in Theorem \ref{thm:SDDE_linear_drift}, the results of Theorem \ref{thm:reg_sol_map} and Proposition \ref{prop:linearization_eq} hold for equation \eqref{drift}, too.
\end{corollary}

\begin{proof}
 We can rewrite the equation \eqref{drift} as
 \begin{align}\label{eqn:reform_eq}
\begin{split}
dy_{t} &= \tilde{\sigma}(y_{t},y_{t-r})d \tilde{\mathbf{X}}_{t}\\
y_{s} &= \xi_{s}, \ \ \ -r\leqslant s\leqslant 0
\end{split}
\end{align}
where $\tilde{\sigma} := (B,\sigma)$ and $\tilde{\mathbf{X}}$ is the delayed rough path obtained from $\mathbf{X}$ by including $t \mapsto t$ as a smooth component, cf. \cite[Section 9.4]{FV10}. Note that $\tilde{\sigma}$ has the same smoothness as $\sigma$. Fixing an initial condition $\xi$ and a neighbourhood around it, we can assume that $\tilde{\sigma}$ is bounded for these initial conditions by replacing the unbounded $\tilde{\sigma}$ by a version which is compactly supported in the region where the respective solutions take their values. Therefore, we can directly apply Theorem \ref{thm:reg_sol_map} and Proposition  \ref{prop:linearization_eq} to \eqref{eqn:reform_eq}.
\end{proof}

We finally give some bounds for the solution to the linearized equation. Since the proofs are a bit technical, we decided to put them in the appendix.

\begin{theorem}\label{THM}
Assume $ \sigma \in C^{3}_{b} $. Then the solution of \eqref{eqn:rough_delay} is differentiable and if $Dy^{\xi}[\tilde{\xi}]$ denotes the derivative at $\xi$ in the direction $\tilde{\xi}$, we have the bound
\begin{align}\label{second-estimate}
\big{\Vert} Dy^{\xi}[\tilde{\xi}]\big{\Vert}_{\mathscr{D}_{X}^{\beta}[0,r]}\leqslant\Vert\tilde{\xi}\Vert_{\mathscr{D}_{X}^{\beta}[-r,0]}\exp[Q(A,\Vert\xi\Vert_{\mathscr{D}_{X}^{\beta}[-r,0]})]
\end{align}
where $ Q $ is a polynomial and $ A = 1 + \Vert \mathbf{X} \Vert_{\gamma,[0,r]}$. If $ \sigma \in C^{4}_{b} $, we have the same result for equation \eqref{drift}.
\end{theorem}
\begin{proof}
Cf. appendix.
\end{proof}

\begin{theorem}\label{TTH}
Under the same assumptions as in Theorem \ref{THM},
\begin{align}\label{derivative-A}
\big{\Vert} Dy^{\xi}[\eta] - Dy^{\tilde{\xi}}[\eta]\big{\Vert}_{\mathscr{D}_{X}^{\beta}[0,r]}\leqslant\Vert\xi - \tilde{\xi}\Vert_{\mathscr{D}_{X}^{\beta}[-r,0]} \Vert\eta\Vert_{\mathscr{D}_{X}^{\beta}[-r,0]} \exp\big{[} P(A,\Vert\xi\Vert_{\mathscr{D}_{X}^{\beta}[-r,0]},\Vert\xi-\tilde{\xi}\Vert_{\mathscr{D}_{X}^{\beta}[-r,0]})\big{]}
\end{align}
for a polynomial $P$.
\end{theorem}
 \begin{proof}
 Cf. appendix.
 \end{proof}
 
  \begin{remark}
Note that since $ P $ is a polynomial, we can find a polynomial $ \tilde{P} $ and an increasing function $ \tilde{Q} $ such that also
 \begin{align}\label{PO}
    \begin{split}
    \big{\Vert} Dy^{\xi}[\eta] - Dy^{\tilde{\xi}}[\eta]\big{\Vert}_{\mathscr{D}_{X}^{\beta}[0,r]} &\leqslant \Vert\xi - \tilde{\xi}\Vert_{\mathscr{D}_{X}^{\beta}[-r,0]} \Vert\eta\Vert_{\mathscr{D}_{X}^{\beta}[-r,0]} \exp\big{[} \tilde{P}(A,\Vert\xi\Vert_{\mathscr{D}_{X}^{\beta}[-r,0]})\big{]} \\
    &\qquad \times \exp\big{[} \tilde{Q}(\Vert\xi-\tilde{\xi}\Vert_{\mathscr{D}_{X}^{\beta}[-r,0]})\big{]}
    \end{split}
\end{align}
holds.
 \end{remark}
 
 \begin{remark}\label{remark:additional_drift}
 If $f \colon W^2 \to W$ has the same smoothness as $\sigma$ and is bounded with bounded derivatives, the equation
 \begin{align}\label{eqn:general_drift}
 \begin{split}
  dy_{t} &= B(y_t,y_{t-r})\, dt + f(y_t,y_{t-r})\, dt + \sigma (y_{t},y_{t-r})\, d \mathbf{X}_{t}\\
  y_{s} &= \xi_{s}, \ \ \ -r\leqslant s\leqslant 0
\end{split}
\end{align}
with initial condition $\xi \in \mathscr{D}_{X}^\beta([-r , 0],W)$ has a unique solution and all results in this section hold for \eqref{eqn:general_drift}, too, where the constants will now depend on $f$ as well. As in the proof of Corollary \ref{corollary:results_linear_true}, this just follows by including $t \mapsto t$ as a smooth component of $\mathbf{X}$ and viewing $(f,\sigma)$ as an element in $C^4_b(W^2,L(\R \oplus U,W))$.
\end{remark}

\section{Invariant manifolds for random rough delay equations}\label{sec:inv_mfd_rrde}

Let $B \colon W^2 \to W$ be a linear map and $\sigma \in C^3_b$ resp. $\sigma \in C^4_b$ in the case when $B \neq 0$. Our goal is to study invariant manifolds for the solution to stochastic delay differential equations of the form
\begin{align}\label{eqn:SDDE_drift}
 dy_t = B(y_t,y_{t-r})\, dt + \sigma(y_t,y_{t-r})\, \star dB_t(\omega)
\end{align}
where $\star dB(\omega)$ can be either the It\=o- or the Stratonovich differential. As already pointed out in \cite[Section 2]{GRS19}, it is equivalent to study the random rough delay equation
\begin{align}\label{eqn:random_rough_delay}
 dy_t = B(y_t,y_{t-r})\, dt + \sigma(y_t,y_{t-r})\, d\mathbf{X}_t(\omega)
\end{align}
where $\mathbf{X}$ is either $\mathbf{B}^{\mathrm{It\bar{o}}}$ or $\mathbf{B}^{\mathrm{Strat}}$, defined, using the It\=o integral, as
\begin{align*}
 \mathbf{B}_{s,t}^{\text{It\=o}} = \left( B_{s,t}, \mathbb{B}_{s,t}^{\text{It\=o}}, \mathbb{B}_{s,t}^{\text{It\=o}}(-r) \right) := \left(B_t - B_s, \int_s^t (B_u - B_s)\, \otimes dB_u, \int_s^t (B_{u - r} - B_{s - r})\, \otimes d B_u \right)
\end{align*}
resp. 
\begin{align*}
  \mathbf{B}_{s,t}^{\text{Strat}} =  \left( B_{s,t}, \mathbb{B}_{s,t}^{\text{It\=o}} + \frac{1}{2}(t-s) I_d, \mathbb{B}_{s,t}^{\text{It\=o}}(-r) \right).
 \end{align*}
 Recall that we could also add a smooth drift term to \eqref{eqn:random_rough_delay} as explained in Remark \ref{remark:additional_drift}, but we will not do so in the sequel for the sake of clarity.

 Using the same cut-off argument as in the proof to Corollary \ref{corollary:results_linear_true}, we can deduce from \cite[Theorem 1.13]{GRS19} that the solution to \eqref{eqn:random_rough_delay} induces a semi-flow $\phi$ on the spaces of controlled paths. From \cite[Theorem 3.7]{GRS19}, we can assume that there is an ergodic metric dynamical system $(\Omega,\mathcal{F},\P,(\theta_t)_{t \in \R})$ on which $\mathbf{B}^{\mathrm{It\bar{o}}}$ and $\mathbf{B}^{\mathrm{Strat}}$ are defined and satisfy the cocyle property. More generally, from now on, we will consider an arbitrary \emph{delayed $\gamma$-rough path cocycle $\mathbf{X}$} which drives the equation \eqref{eqn:random_rough_delay}, cf. \cite[Definition 3.1]{GRS19}. With \cite[Theorem 3.12]{GRS19}, we can deduce that $\varphi(n, \omega, \cdot) := \phi(0,nr,\omega, \cdot)$ is a continuous map 
 \begin{align*}
\varphi(n, \omega, \cdot) \colon \mathscr{D}_{X(\omega)}^{\alpha,\beta}([-r , 0],W) \to \mathscr{D}_{X(\theta_{nr} \omega)}^{\alpha,\beta}([-r , 0],W)
\end{align*}
satisfying the cocycle property
\begin{align}
\varphi(n+m , \omega, \cdot) = \varphi(n, \theta_{mr} \omega, \cdot) \circ \varphi(m, \omega, \cdot)
\end{align}
for every $n, m \in \mathbb{N}_0$ with parameters $\frac{1}{3} < \alpha < \beta < \frac{1}{2}$. From Corollary \ref{corollary:results_linear_true}, the cocycle is differentiable. Set $ \theta^{n} := \theta_{nr} $, $\theta := \theta^1$ and assume that 
\begin{align}\label{eqn:alpha_beta}
\frac{(1-\alpha)(\frac{1}{2}-\beta)}{(1-\beta)(1-2 \alpha)}<\beta -\alpha.
\end{align}
Then by \cite[Proposition 3.15]{GRS19}, $ \lbrace\mathscr{D}_{X(\omega)}^{\alpha ,\beta}([-r,0],W))\rbrace_{\omega\in\Omega}$ constitutes a measurable field of Banach spaces, and the cocycle $\varphi$ defined on the discrete metric dynamical system $(\Omega,\mathcal{F},\P,\theta)$ acts on it, cf. \cite[Theorem 3.17]{GRS19}.

\subsection{Random fixed points and formulation of the main theorems}

In order to deduce the existence of invariant manifolds, we aim to linearize the equation \eqref{eqn:random_rough_delay} around random fixed points which we define now.

\begin{definition}\label{def:stat_traj}
   Let $\varphi$ be a cocycle defined on a metric dynamical system $(\Omega,\mathcal{F},\P,\theta)$ acting on a measurable field of Banach spaces $\{E_{\omega}\}_{\omega \in \Omega}$.  A map $Y:\Omega\longrightarrow \prod_{\omega\in\Omega}E_{\omega}$ is called \emph{stationary trajectory} if the following properties are satisfied:
\begin{itemize}
\item[(i)]$Y_{\omega}\in E_{\omega}$,
\item[(ii)]$\varphi(n, \omega,Y_{\omega}) = Y_{\theta^{n}\omega}$ and
\item[(iii)]$\omega\rightarrow\Vert Y_{\omega}\Vert_{E_{\omega}} $ is measurable.
\end{itemize}
\end{definition}

We aim to apply the Multiplicative Ergodic Theorem in \cite{GR19} to the linearization of \eqref{eqn:random_rough_delay} around a random fixed point. The next lemma gives a sufficient condition under which this can be done.

\begin{lemma}\label{lemma:suff_cond_MET}
    Assume that the cocycle induced by \eqref{eqn:random_rough_delay} admits a stationary trajectory $Y$ and that
\begin{align*}
Q(A_{\omega},\Vert Y_{\omega}\Vert)\in L^{1}(\Omega)
\end{align*}
holds for the polynomial $Q$ obtained in Theorem \ref{THM} where $ A_{\omega} = 1 + \Vert \mathbf{X}(\omega) \Vert_{\gamma,[0,r]}$. Then $ \psi^{n}_{\omega} := D_{Y_{\omega}}\varphi(n,\omega,\cdot) $ defines a compact linear cocycle acting on the measurable field of Banach spaces $ \lbrace\mathscr{D}_{X(\omega)}^{\alpha ,\beta}([-r,0],W))\rbrace_{\omega\in\Omega}$ and the semi-invertible Mutliplicative Ergodic Theorem \cite[Theorem 1.20]{GR19} holds true.
 \end{lemma}

 \begin{proof}
 It is straightforward to check that $\psi$ satisfies the cocycle property. We need to verify \cite[Assumption 1.1]{GR19} which also implies the measurability condition \eqref{eqn:measurability_cocycle}. The proof of \cite[Assumption 1.1]{GR19} is very similar to the proof of \cite[Theorem 3.17]{GRS19} using that $\psi$ solves a (non-autnonomous) linear delay equation, cf. Proposition \ref{prop:linearization_eq} resp. Corollary \ref{corollary:results_linear_true}, so we decided to omit it here. Compactness follows as in the proof of \cite[Proposition 1.12]{GRS19}. From our assumption and Theorem \ref{THM}, it follows that $\log^+ \|\psi^1 \|$ is integrable. Therefore, all conditions of  \cite[Theorem 1.20]{GR19} are indeed satisfied. 
 \end{proof}
 
 From now on, we assume that the conditions of Lemma \ref{lemma:suff_cond_MET} are satisfied. Let $\tilde{\Omega}$ denote the $\theta$-invariant set of full measure provided in  \cite[Theorem 1.20]{GR19}.

 \begin{definition}\label{stable-dfn}
    Let $\lbrace ...<\mu_{j} < \mu_{j-1}<...<\mu_{1} \rbrace \in [-\infty,\infty)$ be the Lyapounov spectrum of $\psi$ provided by the MET \cite[Theorem 4.17]{GRS19} and let $\{H^i_{\omega}\}_{i \in \N}$ be the fast growing subspaces provided by the semi-invertible MET \cite[Theorem 1.20]{GR19}. Recall the splitting
    \begin{align*}
     \mathscr{D}_{X(\omega)}^{\alpha ,\beta}([-r,0],W)) = H^1_{\omega} \oplus \cdots \oplus H^n_{\omega} \oplus F_{\mu_{n+1}}(\omega)
    \end{align*}
    for every $n \in \N_0$ and $\omega \in \tilde{\Omega}$ with $F_{\mu}(\omega)$ defined as in \cite[Theorem 4.17]{GRS19}. Set $ \mu_{j_{0}} := \max \lbrace \mu_{j} : \mu_{j} < 0 \rbrace $ and $ \mu_{j_{0}} := -\infty $ if all $ \mu_{j} $ for which $\mu_j \neq -\infty$ are nonnegative. We define the \emph{stable subspace}
\begin{align*}
S_{\omega} := F_{\mu_{j_{0}}}(\omega)
\end{align*}
for $\omega \in \tilde{\Omega}$. Similarly, if $ \mu_{1}>0 $, set $k_{0} := \min \lbrace k :\mu_{k} > 0\rbrace $ and define the \emph{unstable subspace}
\begin{align*}
{U}_{\omega}:=\oplus_{1\leqslant i\leqslant k_{0}}H^{i}_{\omega}
\end{align*}
for $\omega \in \tilde{\Omega}$. If $ \mu_{1} \leq 0 $, we set $U_{\omega} := \{0\}$.

\end{definition}

From both METs \cite[Theorem 4.17]{GRS19} and  \cite[Theorem 1.20]{GR19}, we know that 
\begin{align*}
 \operatorname{dim}[\mathscr{D}_{X(\omega)}^{\alpha ,\beta}([-r,0],W)) / S_{\omega}] < \infty \quad \text{and} \quad \operatorname{dim}[U_{\omega}] < \infty
\end{align*}
for every $\omega \in \tilde{\Omega}$ and that the dimension does not depend on $\omega$. Note also that
\begin{align*}
 \mathscr{D}_{X(\omega)}^{\alpha ,\beta}([-r,0],W)) = U_{\omega} \oplus S_{\omega}
\end{align*}
in the case where all Lyapounov exponents are nonzero.

Now we are ready to state our main results of this section. Note that they are basically reformulations of the abstract stable and unstable manifold theorems in \cite{GR19}, but we decided to give a full statement here for the readers convencience. We start with the stable case.

\begin{theorem}[Local stable manifolds]\label{thm:stable_manifold}
Let $\mathbf{X}$ be a delayed $\gamma$-rough path cocycle defined on an ergodic metric dynamical system $(\Omega,\mathcal{F},\P,(\theta_t)_{t \in \R})$ and let $\frac{1}{3} < \alpha < \beta < \gamma < \frac{1}{2}$ be such that \eqref{eqn:alpha_beta} holds. Assume $\sigma\in C^{3}_{b}$ resp. $\sigma\in C^{4}_{b}$ in the case $B \neq 0$. Assume also that the cocycle $\varphi$ induced by \eqref{eqn:random_rough_delay} admits a stationary trajectory  $Y$ for which
\begin{align}\label{eqn:inegrability_cond_stable}
\tilde{P}(A_{\omega}, \Vert Y_{\omega}\Vert)\in L^{1}(\Omega) \ \ \text{and}\ \ Q(A_{\omega},\Vert Y_{\omega}\Vert)\in L^{1}(\Omega)
\end{align}
    where $A_{\omega} = 1 + \Vert \mathbf{X}(\omega) \Vert_{\gamma,[0,r]} $, $\tilde{P}$ is the polynomial in \eqref{PO} and $Q$ is the polynomial in \eqref{second-estimate}. Then there is a $\theta$-invariant set of full measure $\tilde{\Omega}$ and a family of immersed submanifolds $ S^{\upsilon}_{loc}(\omega)$ of $\mathscr{D}_{X(\omega)}^{\alpha ,\beta}([-r,0],W))$, $ 0<\upsilon<-\mu_{j_{0}} $ and $\omega \in \tilde{\Omega}$, satisfying in the following properties for every $\omega \in \tilde{\Omega}$:
%
%
%

\begin{itemize}
\item[(i)] There are random variables $ \rho_{1}^{\upsilon}(\omega), \rho_{2}^{\upsilon}(\omega)$, positive and finite on $\tilde{\Omega}$, for which
\begin{align}\label{eqn:rho_temp}
 \liminf_{p \to \infty} \frac{1}{p} \log \rho_i^{\upsilon}(\theta^p \omega) \geq 0, \quad i = 1,2
\end{align}
and such that
\begin{align*}
\big{\lbrace} \xi \in \mathscr{D}_{X(\omega)}^{\alpha ,\beta} \, :\, \sup_{n\geqslant 0}\exp(n\upsilon)\Vert \varphi(n, \omega, \xi) - Y_{\theta^{n}\omega}\Vert &<\rho_{1}^{\upsilon}(\omega)\big{\rbrace}\subseteq S^{\upsilon}_{loc}(\omega)\\&\subseteq \big{\lbrace} \xi \in \mathscr{D}_{X(\omega)}^{\alpha ,\beta} \, :\, \sup_{n\geqslant 0}\exp(n\upsilon)\Vert\varphi(n,\omega,\xi) - Y_{\theta^{n}\omega}\Vert<\rho_{2}^{\upsilon}({\omega})\big{\rbrace}.
\end{align*}
\item[(ii)]
\begin{align*}
 T_{Y_{\omega}}S^{\upsilon}_{loc}(\omega) = S_{\omega}.
\end{align*}
\item[(iii)] For $ n\geqslant N(\omega) $,
\begin{align*}
\varphi(n,\omega,S^{\upsilon}_{loc}(\omega))\subseteq S^{\upsilon}_{loc}(\theta^{n}\omega).
\end{align*}
\item[(iv)]For $ 0<\upsilon_{1}\leqslant\upsilon_{2}< - \mu_{j_{0}} $,
\begin{align*}
S^{\upsilon_{2}}_{loc}(\omega)\subseteq S^{\upsilon_{1}}_{loc}(\omega).
\end{align*}
Also for $n\geqslant N(\omega) $,
\begin{align*}
\varphi(n,\omega,S^{\upsilon_{1}}_{loc}(\omega)) \subseteq S^{\upsilon_{2}}_{loc}(\theta^{n}(\omega))
\end{align*}
 and consequently for $ \xi \in S^{\upsilon}_{loc}(\omega) $,
\begin{align}\label{eqn:contr_char}
    \limsup_{n\rightarrow\infty}\frac{1}{n}\log\Vert\varphi(n,{\omega}, \xi) - Y_{\theta^{n}\omega}\Vert\leqslant  \mu_{j_{0}}.
\end{align}
\item[(v)] 
\begin{align*}
\limsup_{n\rightarrow\infty} \frac{1}{n} \log\bigg{[}\sup\bigg{\lbrace}\frac{\Vert\varphi(n,{\omega},\xi) - \varphi(n,{\omega},\tilde{\xi}) \Vert }{\Vert \xi - \tilde{\xi} \Vert},\ \ \xi \neq\tilde{\xi},\  \xi,\tilde{\xi} \in S^{\upsilon}_{loc}(\omega) \bigg{\rbrace}\bigg{]}\leqslant \mu_{j_{0}}.
\end{align*}
\end{itemize}

\end{theorem}
\begin{proof}
Set $E_{\omega} := \mathscr{D}_{X(\omega)}^{\alpha ,\beta}([-r,0],W)$. In Lemma \ref{lemma:suff_cond_MET}, we saw that our assumptions imply that $\psi^{n}_{\omega} = D_{Y_{\omega}}\varphi(n,\omega,\cdot)$ defines a compact linear cocycle acting on the measurable field of Banach spaces $ \lbrace E_{\omega} \rbrace_{\omega\in\Omega}$, that \cite[Assumption 1.1]{GR19} holds and that $\log^+ \|\psi^1 \| \in L^1(\Omega)$. In view of \cite[Theorem 2.10]{GR19}, it therefore suffices to check the condition \cite[Equation $(2.5)$]{GR19}. Set
 \begin{align*} 
  P_{\omega} : E_{\omega} &\to E_{\theta\omega }\\
  \xi &\mapsto \varphi(1,{\omega}, Y_{\omega} + \xi) - \varphi(1,{\omega}, Y_{\omega}) - \psi^{1}_{\omega}(\xi).
 \end{align*}
 Then from Theorem \ref{TTH},
 \begin{align*}
 \Vert P_{\omega}(\xi)-P_{\omega}(\tilde{\xi})\Vert\leqslant (\Vert\xi\Vert +\Vert\tilde{\xi}\Vert)\exp[\tilde{Q}(\Vert\xi \Vert +\Vert\tilde{\xi}\Vert)]\exp[\tilde{P}(A_{\omega},\Vert Y_{\omega}\Vert)]\ \Vert\xi - \tilde{\xi}\Vert
 \end{align*}
 where $\tilde{P}$ is the polynomial from \eqref{PO} and $\tilde{Q}$ is an increasing function. By Birkhoff's Ergodic Theorem,
 \begin{align*}
 \lim_{n\rightarrow\infty}\frac{1}{n}\tilde{P}(A_{\theta^{n}\omega},\Vert Y_{\theta^{n}\omega}\Vert) = 0
 \end{align*}
 almost surely. Therefore, \cite[Equation $(2.5)$]{GR19} is indeed satisfied and the result follows from  \cite[Theorem 2.10]{GR19}.
\end{proof}

Next, we formulate the result for unstable manifolds.

\begin{theorem}[Local unstable manifolds]\label{thm:unstable_manifold}
Assume the same setting as in Theorem \ref{thm:stable_manifold}. Furthermore, assume that $ \mu_1 > 0 $ holds for the first Lyapunov exponent. Set $\varsigma := \theta^{-1}$. Then there is a $\theta$-invariant set of full measure $\tilde{\Omega}$ and a family of immersed submanifolds $ U^{\upsilon}_{loc}(\omega)$ of $\mathscr{D}_{X(\omega)}^{\alpha ,\beta}([-r,0],W))$, $ 0<\upsilon< \mu_{k_{0}} $ and $\omega \in \tilde{\Omega}$, satisfying in the following properties for every $\omega \in \tilde{\Omega}$:

\begin{itemize}
\item[(i)] There are random variables $ \tilde{\rho}_{1}^{\upsilon}(\omega), \tilde{\rho}_{2}^{\upsilon}(\omega)$, positive and finite on $\tilde{\Omega}$, for which
\begin{align*}
 \liminf_{p \to \infty} \frac{1}{p} \log \tilde{\rho}_i^{\upsilon}(\varsigma^p \omega) \geq 0, \quad i = 1,2
\end{align*}
and such that
\begin{align*}
&\bigg{\lbrace} \xi_{\omega} \in \mathscr{D}_{X(\omega)}^{\alpha ,\beta}\, :\, \exists \lbrace \xi_{\varsigma^{n}\omega}\rbrace_{n\geqslant 1} \text{ s.t. } \varphi({m},{\varsigma^{n}\omega},\xi_{\varsigma^{n}\omega}) = \xi_{\varsigma^{n-m}\omega} \text{ for all } 0 \leq m \leq n \text{ and }\\ 
&\quad \sup_{n\geqslant 0}\exp(n\upsilon)\Vert \xi_{\varsigma^{n}\omega} - Y_{\varsigma^{n}\omega} \Vert < \tilde{\rho}_{1}^{\upsilon}(\omega)\bigg{\rbrace} \subseteq U^{\upsilon}_{loc}(\omega) \subseteq \bigg{\lbrace}  \xi_{\omega} \in \mathscr{D}_{X(\omega)}^{\alpha ,\beta}\, :\, \exists \lbrace \xi_{\varsigma^{n}\omega}\rbrace_{n\geqslant 1} \text{ s.t. } \\
&\qquad \varphi({m},{\varsigma^{n}\omega},\xi_{\varsigma^{n}\omega} ) = \xi_{\varsigma^{n-m}\omega} \text{ for all } 0 \leq m \leq n \text{ and }  \sup_{n\geqslant 0}\exp(n\upsilon)\Vert \xi_{\varsigma^{n}\omega} - Y_{\varsigma^{n}\omega}\Vert <\tilde{\rho}_{2}^{\upsilon}(\omega)\bigg{\rbrace}.
\end{align*}

\item[(ii)]
\begin{align*}
 T_{Y_{\omega}}U^{\upsilon}_{loc}(\omega) = {U}_{\omega}.
\end{align*}

\item[(iii)] For $ n\geqslant N(\omega) $,
\begin{align*}
U^{\upsilon}_{loc}(\omega)\subseteq \varphi({n},{\varsigma^{n}\omega},U^{\upsilon}_{loc}(\varsigma^{n}\omega)).
\end{align*}
\item[(iv)] For $ 0<\upsilon_{1} \leqslant \upsilon_{2} < \mu_{k_{0}} $,
\begin{align*}
U^{\upsilon_{2}}_{loc}(\omega)\subseteq U^{\upsilon_{1}}_{loc}(\omega).
\end{align*}
Also for $n\geqslant N(\omega) $,
\begin{align*}
U^{\upsilon_{1}}_{loc}(\omega)\subseteq \varphi({n},{\varsigma^{n}\omega},U^{\upsilon_{2}}_{loc}(\varsigma^{n}\omega))
\end{align*}
and consequently for $ \xi_{\omega}\in U^{\upsilon}_{loc}(\omega) $, 
\begin{align*}
\limsup_{n\rightarrow\infty}\frac{1}{n}\log\Vert \xi_{\varsigma^{n}\omega} - Y_{\varsigma^{n}\omega}\Vert\leqslant -\mu_{k_{0}}.
\end{align*}
\item[(v)] 
\begin{align*}
\limsup_{n\rightarrow\infty}\frac{1}{n}\log\bigg{[}\sup\bigg{\lbrace}\frac{\Vert \xi_{\varsigma^{n}\omega} - \tilde{\xi}_{\varsigma^{n}\omega}\Vert }{\Vert \xi_{\omega}-\tilde{\xi}_{\omega}\Vert},\ \ \xi_{\omega}\neq\tilde{\xi}_{\omega},\  \xi_{\omega},\tilde{\xi}_{\omega}\in U^{\upsilon}_{loc}(\omega) \bigg{\rbrace}\bigg{]}\leqslant -\mu_{k_{0}}.
\end{align*}
\end{itemize}
\end{theorem}

\begin{proof}
    Follows from \cite[Theorem 2.17]{GR19}.
\end{proof}

\begin{remark}
 \begin{itemize}
  \item[(i)] In both Theorems \ref{thm:stable_manifold} and \ref{thm:unstable_manifold}, the assumption $\sigma \in C^3$ implies that the cocycle $\varphi$ is differentiable. Higher order smoothness of $\sigma$ will lead to higher order differentiability of $\varphi$, cf. Theorem \ref{thm:reg_sol_map}. As a consequence, we obtain higher order smoothness of the stable and unstable manifolds. In fact, $\varphi \in C^m$ implies that  $S^{\upsilon}_{loc}(\omega)$ resp. $ U^{\upsilon}_{loc}(\omega)$ are almost surely locally $C^{m-1}$, cf. \cite[Remark 2.11 and 2.18]{GR19}.
  
  \item[(ii)] If all Lyapunov exponents are non-zero, the stationary trajectory $Y$ is called \emph{hyperbolic}. In this case, the submanifolds $S^{\upsilon}_{loc}(\omega)$ and $U^{\upsilon}_{loc}(\omega)$ are \emph{transversal}, i.e.
  \begin{align*}
   \mathscr{D}_{X(\omega)}^{\alpha ,\beta} = T_{Y_{\omega}} S^{\upsilon}_{loc}(\omega) \oplus T_{Y_{\omega}} U^{\upsilon}_{loc}(\omega)
  \end{align*}
  almost surely.

 \end{itemize}

\end{remark}

\subsection{Examples}\label{subsec:examles}

We will now discuss examples of stochastic delay equations for which we can apply our results. First, we will consider the case of $0$ being a deterministic fixed point for the cocycle.

\begin{proposition}
  Let $\mathbf{X}$ be a delayed $\gamma$-rough path cocycle defined on an ergodic metric dynamical system $(\Omega,\mathcal{F},\P,(\theta_t)_{t \in \R})$ and let $\frac{1}{3} < \alpha < \beta < \gamma < \frac{1}{2}$ be such that \eqref{eqn:alpha_beta} holds. Assume $\sigma\in C^{3}_{b}$ resp. $\sigma\in C^{4}_{b}$ in the case $B \neq 0$ and that
  \begin{align*}
   \sigma(0,0) = \sigma_x(0,0) = \sigma_y(0,0) = 0.
  \end{align*}
  Then $Y \equiv 0$ is a stationary trajectory for the cocycle $\varphi$ induced by 
   \begin{align}\label{eqn:random_rough_delay_det_fp}
      dy_t = B(y_t,y_{t-r})\, dt + \sigma(y_t,y_{t-r})\, d\mathbf{X}_t(\omega).
  \end{align}
  If
\begin{align}\label{eqn:inegrability_cond_prop_fp}
\tilde{P}(A_{\omega}, 0 )\in L^{1}(\Omega) \ \ \text{and}\ \ Q(A_{\omega}, 0)\in L^{1}(\Omega)
\end{align}
    where $A_{\omega} = 1+\Vert \mathbf{X}(\omega) \Vert_{\gamma,[0,r]}$, $\tilde{P}$ is the polynomial in \eqref{PO} and $Q$ is the polynomial in \eqref{second-estimate}, the integrability condition of Theorem \ref{thm:stable_manifold} and Theorem \ref{thm:unstable_manifold} is satisfied and yields the existence of local stable and unstable manifolds around $0$. In particular, the result holds for $\mathbf{X}$ being $\mathbf{B}^{\mathrm{It\bar{o}}}$ or $\mathbf{B}^{\mathrm{Strat}}$.
\end{proposition}

\begin{proof}
 From
 \begin{align*}
  \int_0^t \sigma(y_s,y_{s-r})\, d\mathbf{X}_s(\omega) &= \lim_{|\Pi| \to 0} \sum_{t_j \in \Pi} \sigma(y_{t_j},y_{t_j - r})X_{t_{j},t_{j+1}} +  \sigma_x(y_{t_j},y_{t_j - r}) \sigma(y_{t_j},y_{t_j - r}) \mathbb{X}_{t_{j},t_{j+1}} \\
  &\quad + \sigma_y(y_{t_j},y_{t_j - r}) \sigma(y_{t_j},y_{t_j - r}) \mathbb{X}_{t_{j},t_{j+1}}(-r),
 \end{align*}
 it follows that $Y \equiv 0$ is a solution to \eqref{eqn:random_rough_delay_det_fp} and therefore a stationary trajectory in the sense of Definition \ref{def:stat_traj}. In the case of $\mathbf{X}$ being $\mathbf{B}^{\mathrm{It\bar{o}}}$ or $\mathbf{B}^{\mathrm{Strat}}$, the norm of the delayed rough path cocycle has moments of any order, cf. \cite[Proposition 2.2]{GRS19}, therefore condition \eqref{eqn:inegrability_cond_prop_fp} is satisfied.
\end{proof}

Next, we propose a condition under which \eqref{drift} admits a random stationary trajectory $Y$. Let $B$ be a two-sided Brownian motion defined on a probability space $(\Omega,\mathcal{F},\P)$ adapted to two-parameter filtration $(\mathcal{F}^t_s)_{s \leq t}$ (cf. \cite[Section 2.3.2]{Arn98}). Consider
\begin{align}\label{stationary}
\begin{split}
    dy_{t} &= C y_t\, dt + \sigma(y_{t},y_{t-r})d B_{t}\\
    y_{s} &= \xi_{s}, \ \ \ -r\leqslant s\leqslant 0
\end{split}
\end{align}
 as a classical stochastic delay differential equation in It\=o sense where $C \colon W \to W$ is a linear map. Assume that $ \sigma $ is a bounded Lipschitz function with Lipschitz constant $ L $ and let all the eigenvalues of $C$ be negative. Consequently, there exist $ M,\lambda >0 $ such that for every $ t>0 $,
\begin{align}\label{eqn:lambda_M}
    |\exp(tC)| \leqslant M\exp(-\lambda t).
\end{align}
Set $\mathcal{F}^t_{-\infty} := \sigma( \cup_{s \leq t} \mathcal{F}_s^t)$. A stochastic process $y \colon \R \to W$ is called \emph{$(\mathcal{F}^t_{-\infty})$-adapted} if $y_t$ is $\mathcal{F}^t_{-\infty}$-measurable for every $t \in \R$. In that case for, any continuous, $(\mathcal{F}^t_{-\infty})$-adapted process $y$, the following process is well defined, continuous and $(\mathcal{F}^t_{-\infty})$-adapted:
\begin{align*}
\Gamma(y)(t):=\int_{-\infty}^{t}\exp((t-\tau)C) \sigma (y_{\tau},y_{\tau-r})\, dB_{\tau}.
\end{align*}
By the It\=o isometry,
\begin{align}\label{IT}
\begin{split}
    \E|\Gamma(y)(t)|^{2} &\leq \E \int_{-\infty}^{t}|\exp((t-\tau)C)|^2 |\sigma (y_{\tau},y_{\tau-r})|^{2}\, ds,\\
    \E|\Gamma(y)(t)-\Gamma(\tilde{y})(t)|^{2} &\leq \E  \int_{-\infty}^{t} |\exp((t-\tau)C)|^2 |\sigma(y_{\tau},y_{\tau-r}) - \sigma(\tilde{y}_{\tau},\tilde{y}_{\tau-r})|^{2}\, ds .
\end{split}
\end{align}

\begin{lemma}\label{lemma:ex_stat}
Assume $ \frac{2ML^{2}}{\lambda}<1 $. Then there is a continuous, $(\mathcal{F}^t_{-\infty})$-adapted process $ Y_{t} $ such that for every $ t\in \mathbb{R} $,
\begin{align*}
Y_{t} = \int_{-\infty}^{t}\exp((t-\tau)C) \sigma(Y_{\tau},Y_{\tau-r})\, dB_{\tau}.
\end{align*}
\end{lemma}
\begin{proof}
Set
\begin{align*}
\mathcal{X}:= \left\{ y:\mathbb{R}\rightarrow W\, :\,  y \ \text{is continuous,  $(\mathcal{F}^t_{-\infty})$-adapted and}\ \sup_{t\in\mathbb{R}}(\E|y_{t}|^{2})^{\frac{1}{2}}<\infty  \right\}.
\end{align*}
It can easily be seen that $ \mathcal{X} $ is a Banach space. By \eqref{IT},
\begin{align*}
\Gamma :\mathcal{X}\longrightarrow\mathcal{X}
\end{align*}
is a contraction, so our claim follows from a standard fixed point argument.
\end{proof}

\begin{lemma}\label{eqn:Y_controlled}
 Let $Y$ be the process from Lemma $\ref{lemma:ex_stat}$ and set $Y'_t = \sigma(Y_t,Y_{t-r})$. Then $(Y,Y')$ is almost surely controlled by $B$. Moreover, $\| (Y,Y')\|_{\mathscr{D}_{B}^{\gamma}([a,b],W)} \in L^p(\Omega)$ for every $p > 0$ and every $a<b$.
\end{lemma}

\begin{proof}
From the Burkholder-Davis-Gundy inequality, for every $ m\in\mathbb{N} $ there exists a $ \beta_{2m}\in\mathbb{R}  $ such that
\begin{align}\label{Holder}
    \E |Y_{s,t}|^{2m} \leqslant \beta_{2m}(t-s)^{m}
\end{align}
for every $s < t$. Note that
\begin{align*}
Y_{s,t}-\sigma(Y_{s},Y_{s-r})B_{s,t} &= \int^{s}_{-\infty}\exp{(}(s-\tau)C{)}\big{[}\exp\big{(}(t-s)C{)}-1\big{]}\sigma(Y_{\tau},Y_{\tau -r})\, dB_{\tau}\\
&\quad + \int_{s}^{t}\exp((t-\tau)C)\big{[}\sigma(Y_{\tau},Y_{\tau-r})-\sigma(Y_{s},Y_{s-r})\big{]}\, dB_{\tau} \\
&\quad + \int_{s}^{t}\big{[}\exp((t-\tau)C)-1\big{]}\, dB_{\tau} \, \sigma(Y_{s},Y_{s-r}).
\end{align*}
By the Burkholder-Davis-Gundy inequality and our assumptions, for $\alpha_{2m}\in\mathbb{R}$,
\begin{align*}
    \E\bigg{|}\int^{s}_{-\infty}\exp{(}(s-\tau)C{)}\big{[}\exp\big{(}(t-s)C{)}-1\big{]}\sigma(Y_{\tau},Y_{\tau -r})\, dB_{\tau}\bigg{|}^{2m}\leqslant \alpha_{2m}(t-s)^{2m}
\end{align*}
and
\begin{align*}
    \E\bigg{|}\int_{s}^{t}\big{[}\exp((t-\tau)C)-1\big{]}dB_{\tau} \, \sigma(Y_{s},Y_{s-r}) \bigg{|}^{2m}\leqslant \alpha_{2m}(t-s)^{2m}.
\end{align*}
Using again the Burkholder-Davis-Gundy inequality, H\"older's inequality and  \eqref{Holder}, we obtain that there are constants $\beta_{2m},\gamma_{2m}\in\mathbb{R} $ such that
\begin{align*}
    &\E\bigg{|}\int_{s}^{t}\exp((t-\tau)C)\big{[}\sigma(Y_{\tau},Y_{\tau-r})-\sigma(Y_{s},Y_{s-r})\big{]}\, dB_{\tau}\bigg{|}^{2m} \\
    \leqslant\ &\beta_{2m} \E\ \bigg{|}\int_{s}^{t}(\vert Y_{s,\tau}\vert^{2}+\vert Y_{s-r,\tau-r}\vert^{2})d\tau\bigg{|}^{m}\\
    \leqslant\ &\beta_{2m}(t-s)^{m-1} \E\int_{s}^{t}(\vert Y_{s,\tau}\vert^{2} + \vert Y_{s-r,\tau-r}\vert^{2})^{m}d\tau \leqslant \gamma_{2m}(t-s)^{2m}.
\end{align*}
Consequently, we have shown that for every $m \geq 1$ there are constants $ \tilde{\alpha}_{2m} $ such that
\begin{align*}
    \E|Y_{s,t} - \sigma(Y_{s},Y_{s-r})B_{s,t}|^{2m} \leqslant \tilde{\alpha}_{2m}(t-s)^{2m}
\end{align*}
for every $s<t$. Set $ Y^{\#}_{s,t} := Y_{s,t}-\sigma(Y_{s},Y_{s-r})B_{s,t}$. By a version of Kolmogorov's continuity theorem similar to \cite[Theorem 3.1]{FH14}, we obtain
\begin{align*}
    \Vert Y\Vert_{\gamma;[a,b]} + \Vert Y^{\#}\Vert_{2\gamma;[a,b]} \in L^p(\Omega)
\end{align*}
for every $p > 0$ and $a < b$ from which the result follows.

\end{proof}

\begin{proposition}
 Let $C$ be a linear map with negative eigenvalues only and $\sigma \in C^4_b$. Let $\lambda$ and $M$ be as in \eqref{eqn:lambda_M} and let $L$ be the Lipschitz constant of $\sigma$. Assume $ \frac{2ML^{2}}{\lambda}<1 $. Then there exists a stationary trajectory for the cocycle $\varphi$ induced by
 \begin{align}\label{eqn:stationary_rough}
\begin{split}
    dy_{t} &= C y_t\, dt + \sigma(y_{t},y_{t-r})d \mathbf{B}^{\mathrm{It\bar{o}}}_{t}\\
    y_{s} &= \xi_{s}, \ \ \ -r\leqslant s\leqslant 0
\end{split}
\end{align}
and the integrability condition \eqref{eqn:inegrability_cond_stable} of Theorem \ref{thm:stable_manifold} and Theorem \ref{thm:unstable_manifold} is satisfied.

\end{proposition}

\begin{proof}
 Let $\hat{Y} = (Y,Y')$ be defined as in Lemma \ref{eqn:Y_controlled}. From \cite[Proposition 3.2]{GRS19}, 
 \begin{align*}
  \hat{Y}_{t} = \int_{-\infty}^{t}\exp((t-\tau)C) \sigma(\hat{Y}_{\tau},\hat{Y}_{\tau-r})\, d \mathbf{B}^{\mathrm{It\bar{o}}}_{t}
 \end{align*}
 almost surely for every $t$. Therefore, (i) and (ii) of Definition \ref{def:stat_traj} follow directly. Since
 \begin{align*}
  \| \hat{Y} \|_{\mathscr{D}_{B}^{\beta}([-r,0])} = |Y_{-r}| + |Y'_{-r}| + \sup_{s,t \in [-r,0] \cap \mathbb{Q}, s \neq t} \frac{|Y'_t - Y'_s|}{|t-s|^{\beta}} + \sup_{s,t \in [-r,0] \cap \mathbb{Q}, s \neq t} \frac{|Y_{s,t} - Y'_sB_{s,t}|}{|t-s|^{2\beta}},
 \end{align*}
 measurability of $\omega \mapsto \| \hat{Y}(\omega) \|_{\mathscr{D}_{B(\omega)}^{\beta}([-r,0])}$ follows, too. The integrability condition \eqref{eqn:inegrability_cond_stable} is satisfied due to Lemma \ref{eqn:Y_controlled} and \cite[Proposition 2.2]{GRS19}.

\end{proof}

\begin{remark}
 It is possible to prove directly that the rough differential equation
 \begin{align*}
  \hat{Y}_{t} = \int_{-\infty}^{t}\exp((t-\tau)C) \sigma(\hat{Y}_{\tau},\hat{Y}_{\tau-r})\, d \mathbf{B}^{\mathrm{It\bar{o}}}_{t}
 \end{align*}
 has a fixed point using the standard estimates for the rough integral. However, this would yield a stronger condition than $ \frac{2ML^{2}}{\lambda}<1 $.

\end{remark}

\section*{Appendix}

\begin{proof}[Proof of Theorem \ref{THM}]
We start with equation \eqref{eqn:rough_delay}. From Proposition \ref{prop:linearization_eq}, the derivative of the solution at $ \xi $ in the direction of $ \tilde{\xi} $ satisfies the equation
 \begin{align}\label{equation}
  \begin{split}
Dy^{\xi}[\tilde{\xi}](t) - \tilde{\xi}_0 &= \int_{0}^{t}\big{[}\sigma_{x}(y^{\xi}_{\tau},\xi_{\tau -r})Dy^{\xi}[\tilde{\xi}](\tau)+\sigma_{y}(y^{\xi}_{\tau},\xi_{\tau-r})\tilde{\xi}_{\tau -r}\big{]}d\mathbf{X}_{\tau};\quad t \in [0,r] \\
Dy^{\xi}[\tilde{\xi}](t) &= \tilde{\xi}_t; \quad t \in [-r,0].
  \end{split}
\end{align}
Set $ Z_{\tau}=Dy^{\xi}[\tilde{\xi}](\tau) $ and $ \eta_{t}=\sigma_{x}(y^{\xi}_{t},\xi_{t -r})Z_{t}+\sigma_{y}(y^{\xi}_{t},\xi_{t-r})\tilde{\xi}_{t -r} $. Using a Taylor expansion and the definition of controlled paths, we obtain
\begin{align}\label{decompositon1}
\begin{split}
\eta_{s,t} &= \sigma_{x}(y_{s}^{\xi},\xi_{s-r})Z^{\prime}_{s} X_{s,t} + \big{[} \sigma_{x^2}(y_{s}^{\xi},\xi_{s-r})(y^{\xi})^{\prime}_{s} X_{s,t} + \sigma_{x,y} (y_{s}^{\xi},\xi_{s-r}) \xi^{\prime}_{s-r}X_{s-r,t-r}\big{]}Z_{s} \\
&\quad + \sigma_{y}(y^{\xi}_{s},\xi_{s-r})(\tilde{\xi})^{\prime}_{s-r}X_{s-r,t-r}+\big{[} \sigma_{x,y}(y_{s}^{\xi},\xi_{s-r})(y^{\xi})^{\prime}_{s}X_{s,t} + \sigma_{y^2}(y_{s}^{\xi},\xi_{s-r})\xi^{\prime}_{s-r}X_{s-r,t-r}\big{]}\tilde{\xi}_{s-r} \\
&\quad +\eta^{\#}_{s,t}
\end{split}
\end{align}
where 
\begin{align}\label{decomposition2}
\begin{split}
    \eta_{s,t}^{\#} &= \big{[}\sigma_{x}(y_{t}^{\xi},\xi_{t-r})-\sigma_{x}(y_{s}^{\xi},\xi_{s-r})\big{]}Z_{s,t}+\big{[}\sigma_{y}(y_{t}^{\xi},\xi_{t-r})-\sigma_{y}(y_{s}^{\xi},\xi_{s-r})\big{]}\tilde{\xi}_{s-r,t-r}+\sigma_{x}(y_{s}^{\xi},\xi_{s-r})Z^{\#}_{s,t}\\
&\quad + \sigma_{y}(y_{s}^{\xi},\xi_{s-r})\tilde{\xi}^{\#}_{s,t} + \big{[}\sigma_{x^2}(y_{s}^{\xi},\xi_{s-r})(y^{\xi})^{\#}_{s,t} + \sigma_{x,y}(y_{\xi}^{s},\xi_{s-r})\xi^{\#}_{s-r,t-r}\big{]}Z_{s} + 
\big{[} \sigma_{x,y}(y_{s}^{\xi},\xi_{s-r})(y^{\xi})^{\#}_{s,t}\\
&\quad + \sigma_{y^2}(y_{\xi}^{s},\xi_{s-r})\xi^{\#}_{s-r,t-r}\big{]}\tilde{\xi}_{s-r} + \int_{0}^{1}(1-z)\frac{d^{2}}{dz^{2}}\bigg{[}\sigma_{x}\big{(}z y^{\xi}_{t} + (1-z)y_{s}^{\xi}, z\xi_{t-r} + (1-z)\xi_{s-r}\big{)}\bigg{]}Z_{s}\, dz \\
&\quad + \int_{0}^{1}(1-z)\frac{d^{2}}{dz^{2}}\bigg{[}\sigma_{y}\big{(}z y^{\xi}_{t} + (1-z)y_{s}^{\xi}, z\xi_{t-r} + (1-z)\xi_{s-r}\big{)}\bigg{]}\tilde{\xi}_{s-r}\, dz
\end{split}
\end{align}
and $Z_{s,t} = Z'_s X_{s,t} + Z^{\#}_{s,t}$ with
\begin{align*}
 Z'_s = \sigma_{x}(y^{\xi}_{s},\xi_{s -r})Dy^{\xi}[\tilde{\xi}](s)+\sigma_{y}(y^{\xi}_{s},\xi_{s-r})\tilde{\xi}_{s -r}.
\end{align*}
By \cite[Theorem 1.5]{GRS19}, for a delayed controlled path with decomposition $\eta_{s,t} = \eta_{s}^{1}X_{s,t}+\eta^{2}_{s}X_{s-r,t-r}+\eta^{\#}_{s,t} $, we have for any $w_0 \in W$
\begin{align}\label{Bound}
\begin{split}
&\big{\Vert} w_{0}+\int_{a}^{\cdot}\eta_{\tau}\, d\mathbf{X}_{\tau}\big{\Vert}_{\mathscr{D}_{X}^{\beta}[a,b]}\leqslant\vert w_{0}\vert + \vert\eta_{a}\vert + \Vert\eta\Vert_{\beta;[a,b]} + \sup_{a \leq s < t \leq b} \frac{ \left| \int_s^t \eta_{\tau}\, d\mathbf{X}_{\tau} - \eta_s X_{s,t}\right|}{|t-s|^{2\beta}}
\end{split}
\end{align}
and
\begin{align*}
 \sup_{a \leq s < t \leq b} \frac{ \left| \int_s^t \eta_{\tau}\, d\mathbf{X}_{\tau} - \eta_s X_{s,t}\right|}{|t-s|^{2\beta}} \leq\ &\Vert\eta^{1}\Vert_{\infty;[a,b]}\Vert\mathbb{X}\Vert_{\gamma;[a,b]}(b-a)^{2(\gamma-\beta)} + \Vert\eta^{2}\Vert_{\infty;[a,b]}\Vert\mathbb{X}(-r)\Vert_{\gamma;[a,b]}(b-a)^{2(\gamma-\beta)} \\
 &\quad + M\bigg{[}\Vert\eta^{\#}\Vert_{2\beta;[a,b]}\Vert X\Vert_{\gamma;[a,b]}(b-a)^{\gamma} + \Vert\eta^{1}\Vert_{\beta;[a,b]}\Vert\mathbb{X}\Vert_{2\gamma;[a,b]}(b-a)^{2\gamma-\beta} \\
 &\qquad + \Vert\eta^{2}\Vert_{\beta;[a,b]}\Vert\mathbb{X}(-r)\Vert_{2\gamma;[a,b]}(b-a)^{2\gamma-\beta} \bigg{]}
\end{align*}
for a general constant $M$.
%
%
%
Thanks to our assumptions on $\sigma$, \eqref{decompositon1}, \eqref{decomposition2} and Theorem \ref{thm:bound_rdde},
\begin{align*}
\max \bigg{\lbrace}\Vert\eta^{1}\Vert_{\beta;[a,b]},\Vert\eta^{2}\Vert_{\beta;[a,b]}, \Vert\eta^{\#}\Vert_{2\beta ; [a,b]} \bigg{\rbrace}\leqslant\big{[}\Vert Z\Vert_{\mathscr{D}_{X}^{\beta}[0,r]}+\Vert\tilde{\xi}\Vert_{\mathscr{D}_{X}^{\beta}[-r,0]}\big{]}Q_{1}(A,\Vert\xi\Vert_{\mathscr{D}_{X}^{\beta}[-r,0]})
\end{align*}
and
\begin{align*}
\Vert\eta\Vert_{\beta;[a,b]}\leqslant (b-a)^{\gamma-\beta}\big{[}\Vert Z\Vert_{\mathscr{D}_{X}^{\beta}[0,r]}+\Vert\tilde{\xi}\Vert_{\mathscr{D}_{X}^{\beta}[-r,0]}\big{]} Q_{1}(A,\Vert\xi\Vert_{\mathscr{D}_{X}^{\beta}[-r,0]})
\end{align*} 
for a polynomial $ Q_{1} $. Using this bound in \eqref{equation}, we see that for $0 \leq (n-1)\tau < n \tau \leq r$
\begin{align*}
\Vert Z\Vert_{\mathscr{D}_{X}^{\beta}[(n-1)\tau ,n\tau]} &\leqslant \tau^{\gamma -\beta} \Vert Z\Vert_{\mathscr{D}_{X}^{\beta}[(n-1)\tau ,n\tau]}Q_{2}(A,\Vert\xi\Vert_{\mathscr{D}_{X}^{\beta}[-r,0]})\\
&\quad + \Vert\tilde{\xi}\Vert_{\mathscr{D}_{X}^{\beta}[-r,0]}Q_{2}(A,\Vert\xi\Vert_{\mathscr{D}_{X}^{\beta}[-r,0]})+
\vert Z_{(n-1)\tau}\vert+\vert Z^{\prime}_{(n-1)\tau}\vert
\end{align*}
for a polynomial $Q_2$. Choosing $ \tau $ such that $\tau^{\gamma-\beta}Q_{2}(A,\Vert\xi\Vert_{\mathscr{D}_{X}^{\beta}[-r,0]}) \leq \frac{1}{2} $, we can proceed as in the proof of \cite[Theorem 1.11]{GRS19} to conclude the claimed bound for \eqref{eqn:rough_delay}. The proof for \eqref{drift} is similar.
\end{proof}

\begin{proof}[Proof of Theorem \ref{TTH}]
We will prove the statement for the solution to \eqref{eqn:rough_delay} only, the proof for \eqref{drift} is similar. Set $ Z^{1}_{\tau} := Dy^{\xi}[\eta](\tau) $ and $ Z^{2}_{\tau} := Dy^{\tilde{\xi}}[\eta](\tau)$. From Proposition \ref{prop:linearization_eq},
 \begin{align}\label{A}
 [Z^{1}_{s,t}-Z^{2}_{s,t}] = \int_{s}^{t}\big{[}\sigma_{x}(y_{\tau}^{\xi},\xi_{\tau-r})[Z^{1}_{\tau}-Z^{2}_{\tau}] + B_{\tau}\big{]}\, d\mathbf{X}_{\tau}
 \end{align}
 where 
 \begin{align*}
   B_{\tau} &:= [\sigma_{x}(y_{\tau}^{\xi},\xi_{\tau -r}) - \sigma_{x}(y^{\tilde{\xi}}_{\tau},\tilde{\xi}_{\tau -r})]Z^{2}_{\tau} + [\sigma_{y}(y_{\tau}^{\xi},\xi_{\tau -r})-\sigma_{y}(y^{\tilde{\xi}}_{\tau},\tilde{\xi}_{\tau -r})]\eta_{\tau -r} \\
   &=: B^{1}_{\tau} + B^{2}_{\tau}.
 \end{align*}
 Set $ C_{\tau} := [\sigma_{x}(y_{\tau}^{\xi},\xi_{\tau -r})-\sigma_{x}(y^{\tilde{\xi}}_{\tau},\tilde{\xi}_{\tau -r})] $. By a Taylor expansion,
 \begin{align*}
 C_{s,t} &= \big{[} \sigma_{x^2}(y_{s}^{\xi},\xi_{s-r})(y^{\xi})^{\prime}_{s} - \sigma_{x^2}(y_{s}^{\tilde{\xi}},\tilde{\xi}_{s-r})(y^{\tilde{\xi}})^{\prime}_{s}\big{]}X_{s,t} \\
 &\quad + \big{[} \sigma_{x,y}(y_{t}^{\xi},\xi_{t-r}){\xi}^{\prime}_{s-r} - \sigma_{x,y}(y_{s}^{\tilde{\xi}},\tilde{\xi}_{t-r}){\tilde{\xi}}^{\prime}_{s-r}\big{]}X_{s-r,t-r} \\
 &\quad + \big{[} \sigma_{x^2}(y_{s}^{\xi},\xi_{s-r})(y^{\xi})^{\#}_{s,t} - \sigma_{x^2}(y_{s}^{\tilde{\xi}},\tilde{\xi}_{s-r})(y^{\tilde{\xi}})^{\#}_{s,t}\big{]} + \big{[} \sigma_{x,y}(y_{s}^{\xi},\xi_{s-r}){\xi}^{\#}_{s,t} - \sigma_{x,y}(y_{s}^{\tilde{\xi}},\tilde{\xi}_{s-r}){\tilde{\xi}}^{\#}_{s,t}\big{]}\\
 &\quad + \int_{0}^{1}(1-z)\frac{d^{2}}{dz^{2}}\bigg{[}\sigma_{x}\big{(} z y^{\xi}_{t} + (1-z) y_{s}^{\xi},x\xi_{t-r} + (1-z)\xi_{s-r}\big{)} \\
 &\quad \quad -\sigma_{x}\big{(} z y^{\tilde{\xi}}_{t} 
 + (1-z) y_{s}^{\tilde{\xi}},z\tilde{\xi}_{t-r} + (1-z)\tilde{\xi}_{s-r}\big{)}\bigg{]}\, dz \\
 &=: C^{1}_{s}X_{s,t} + C^{2}_{s,t} X_{s-r,t-r} + C^{\#}_{s,t}.
 \end{align*}
Note that
\begin{align*}
C^{1}_{s,t} &= \int_{0}^{1}\frac{d}{dz} \bigg{[} \sigma_{x^2}\big{(} z y_{t}^{\xi} + (1-z) y_{t}^{\tilde{\xi}}, z \xi_{t-r} + (1-z) \tilde{\xi}_{t-r}\big{)}\\
&\quad \quad - \sigma_{x^2} \big{(} z y_{s}^{\xi} + (1-z) y_{s}^{\tilde{\xi}},z \xi_{s-r} + (1-z) \tilde{\xi}_{s-r}\big{)}\bigg{]}(y^{\xi})^{\prime}_{t}\, dz \\
&\quad + \sigma_{x^2}(y_{t}^{\tilde{\xi}},\tilde{\xi}_{t-r})\big{[}(y^{\xi})^{\prime}_{s,t} - (y^{\tilde{\xi}})^{\prime}_{s,t}\big{]}\\
&\quad + \int_{0}^{1} \frac{d}{dz} \bigg{[} \sigma_{x^2}\big{(} z y_{s}^{\xi} + (1-z) y_{s}^{\tilde{\xi}}, z\xi_{s-r} + (1-z) \tilde{\xi}_{s-r}\big{)}\bigg{]}(y^{\xi})^{\prime}_{s,t}\ dz \\
&\quad + \big{[}\sigma_{x^2}(y_{t}^{\tilde{\xi}},\tilde{\xi}_{t-r}) - \sigma_{x^2} (y_{s}^{\tilde{\xi}},\tilde{\xi}_{s-r})\big{]}\big{[}(y^{\xi})^{\prime}_{s} - (y^{\tilde{\xi}})^{\prime}_{s}\big{]}.
\end{align*}
From Theorem \ref{thm:bound_rdde}, Theorem \ref{THM} and our assumptions on $\sigma$,
\begin{align}
    \max \big{\lbrace} \Vert C^{1}\Vert_{\beta;[0,r]},\Vert C^{1}\Vert_{\infty;[0,r]}\big{\rbrace}\leqslant\Vert\xi-\tilde{\xi}\Vert_{\mathscr{D}_{X}^{\beta}[-r,0]}\exp\big{[}P_{1}(A,\Vert\xi\Vert_{\mathscr{D}_{X}^{\beta}[-r,0]},\Vert\xi-\tilde{\xi}\Vert_{\mathscr{D}_{X}^{\beta}[-r,0]})\big{]}
\end{align}
where $ P_{1} $ is a polynomial. Note that 
\begin{align*}
    B^{1}_{s,t}=[C^{1}_{s}X_{s,t}]Z_{s}^{2}+C_{s}[(Z^{2})^{\prime}_{s}X_{s,t}]+[C^{2}_{s}X_{s-r,t-r}]Z_{s}^{2}+C^{\#}_{s,t}Z^{2}_{s}+C_{s}(Z^{2})^{\#}_{s,t}+C_{s,t}Z^{2}_{s,t}.
\end{align*}
Setting $D_{\tau} = \sigma_{y}(y_{\tau}^{\xi},\xi_{\tau -r})-\sigma_{y}(y^{\tilde{\xi}}_{\tau},\tilde{\xi}_{\tau -r})$, we have the same decomposition for $B^{2}_{\tau} = D_{\tau}\eta_{\tau -r}$ with similar estimates. Using  \cite[Theorem 1.5]{GRS19}, we can deduce that there exists a polynomial $ P_{2} $ such that for every $ [a,b]\in[0,r] $,
\begin{align}\label{B}
\big{\Vert}\int_{a}^{.}B_{\tau}\, d\mathbf{X}_{\tau}\big{\Vert}_{\mathscr{D}_{X}^{\beta}[a,b]}\leqslant\Vert\xi - \tilde{\xi} \Vert_{\mathscr{D}_{X}^{\beta}[-r,0]} \Vert\eta\Vert_{\mathscr{D}_{X}^{\beta}[-r,0]} \exp\big{[}P_{2}(A,\Vert\xi\Vert_{\mathscr{D}_{X}^{\beta}[-r,0]},\Vert\xi-\tilde{\xi}\Vert_{\mathscr{D}_{X}^{\beta}[-r,0]})\big{]}.
\end{align}
By a similar argument as in the proof of Theorem \ref{THM},
\begin{align}\label{C}
    \big{\Vert}\int_{a}^{.}\sigma_{x}(y_{\tau}^{\xi},\xi_{\tau-r})[Z^{1}_{\tau}-Z^{2}_{\tau}]\, d\mathbf{X}_{\tau}\big{\Vert}_{\mathscr{D}_{X}^{\beta}[a,b]}\leqslant (b-a)^{\gamma-\beta}\Vert Z^{2}-Z^{1}\Vert_{\mathscr{D}_{X}^{\beta}[a,b]} P_3(A,\Vert\xi\Vert_{\mathscr{D}_{X}^{\beta}[-r,0]})
\end{align}
for a polynomial $P_3$. Finally from \eqref{A}, \eqref{B} and \eqref{C}, we obtain for $0 \leq (n-1)\tau <n\tau \leq r$
\begin{align*}
\Vert Z^{1}-Z^{2}\Vert_{\mathscr{D}_{X}^{\beta}[(n-1)\tau,n\tau]} &\leqslant \tau^{\gamma -\beta}\Vert Z^{1}-Z^{2}\Vert_{\mathscr{D}_{X}^{\beta}(n-1)\tau,n\tau]} P_3 (A,\Vert\xi\Vert_{\mathscr{D}_{X}^{\beta}[-r,0]}) \\
&\quad + \Vert\xi -\tilde{\xi} \Vert_{\mathscr{D}_{X}^{\beta}[-r,0]}  \Vert\eta\Vert_{\mathscr{D}_{X}^{\beta}[-r,0]} \exp\big{[}P_{2}(A,\Vert\xi\Vert_{\mathscr{D}_{X}^{\beta}[-r,0]},\Vert\xi - \tilde{\xi}\Vert_{\mathscr{D}_{X}^{\beta}[-r,0]})\big{]}\\
&\quad +\vert[Z^{1}-Z^{2}]_{(n-1)\tau}\vert + \vert[Z^{1}-Z^{2}]_{(n-1)\tau}^{\prime}\vert
\end{align*}
Choosing $\tau$ such that $ \tau^{\gamma-\beta}\tilde{Q}(A,\Vert\xi\Vert_{\mathscr{D}_{X}^{\beta}(n-1)\tau,n\tau]}) \leq \frac{1}{2}$, we can again proceed as in the proof of \cite[Theorem 1.11]{GRS19} to obtain the result.
\end{proof}

\subsection*{Acknowledgements}
\label{sec:acknowledgements}

MGV acknowledges a scholarship from the Berlin Mathematical School (BMS). SR is supported by the MATH+ project AA4-2 \textit{Optimal control in energy markets using rough analysis and deep networks}. Work on this paper was started while SR was supported by the DFG via Research Unit FOR 2402. Both authors would like to thank M.~Scheutzow for valuable discussions and comments during the preparation of the manuscript.

\bibliographystyle{alpha}
\bibliography{refs}

\end{document}